\documentclass[reqno, english,11pt]{amsart}
 
\RequirePackage{mathrsfs} \let\mathcal\mathscr
\usepackage{a4wide}
\usepackage{color}

\usepackage{amssymb,amsmath,amsfonts, amsthm}
\usepackage{hyperref}
\usepackage{dsfont}
\usepackage{upgreek}
\usepackage{mathrsfs}
\usepackage{mathabx,yfonts}
\usepackage{enumerate}
\usepackage{enumitem}
\usepackage{graphicx}

\newtheorem{theorem}{Theorem} 
\newtheorem{conjecture}{Conjecture}
\newtheorem{lemma}[theorem]{Lemma}

\theoremstyle{definition}

\newtheorem*{notation}{Notation} 
\newtheorem{remark}{Remark}

\numberwithin{theorem}{section}
\numberwithin{equation}{section}
\numberwithin{remark}{section}

\DeclareSymbolFont{bbold}{U}{bbold}{m}{n}
\DeclareSymbolFontAlphabet{\mathbbold}{bbold}

\newcommand{\Q}{\mathbb{Q}}
\newcommand{\F}{\mathbb{F}}
\newcommand{\N}{\mathbb{N}}
\newcommand{\C}{\mathbb{C}}
\newcommand{\R}{\mathbb{R}}

\newcommand{\Z}{\mathbb{Z}}

\renewcommand{\l}{\left}

\renewcommand{\r}{\right}
\renewcommand{\b}{\mathbf}
\renewcommand{\c}{\mathcal}
\renewcommand{\epsilon}{\varepsilon}

\renewcommand{\leq}{\leqslant}
\renewcommand{\geq}{\geqslant}
\renewcommand{\#}{\sharp}
\renewcommand{\gg}{\ggg}

\renewcommand{\ll}{\lll}

\newcommand{\OO}{\mathcal{O}}
\newcommand{\places}{\Omega_K}
\newcommand{\archplaces}{{\Omega_\infty}}
\newcommand{\vt}{{\boldsymbol{\theta}}}
\newcommand{\vpsi}{{\boldsymbol{\psi}}}
\newcommand{\vv}{\mathbf{v}}

\newcommand{\ve}{\mathbf{e}}
\newcommand{\ideals}{\mathcal{I}_K}
\newcommand{\id}[1]{\mathfrak{#1}}
\newcommand{\ppp}{\id{p}}

\newcommand{\aaa}{\id{a}}
\newcommand{\bbb}{\id{b}}
\newcommand{\ccc}{\id{c}}
\newcommand{\ddd}{\id{d}}
\newcommand{\eee}{\id{e}}
\newcommand{\norm}{\mathfrak{N}}
\newcommand{\where}{\ :\ }
\newcommand{\classrep}{\mathfrak{r}}
\newcommand{\abs}[1]{\left|#1\right|}
\newcommand{\absv}[1]{\left|#1\right|_v}
\newcommand{\vecnorm}[1]{\left\lVert #1 \right\rVert}

\newcommand{\vic}{\underline{\ccc}}
\newcommand{\vib}{\underline{\bbb}}
\newcommand{\via}{\underline{\aaa}}

\newcommand{\one}{{\mathbf{1}}}

\newcommand\qr[2]{\left(\frac{#1}{#2}\right)}
\newcommand\qrp[1]{\qr{#1}{\ppp}}
\newcommand\card{\#}
\newcommand\DD{\mathcal{D}}
\newcommand\www{\id{W}}
\newcommand\BBB{\id{B}}
\newcommand{\locdegv}{{m_v}}
\newcommand\dg{m}

\newcommand\bad{\text{bad}}

\DeclareMathOperator{\res}{Res}

\DeclareMathOperator{\vol}{vol}

\begin{document}

\title[
Generalised divisor sums over number fields
]
{
Generalised divisor sums of binary forms over number fields
}

\author{Christopher Frei}
\address{Technische Universit\"at Graz\\
Institut f\"ur Analysis und Zahlentheorie\\
Kopernikusgasse 24/II, A-8010 Graz, Austria
\\}
\email{frei@math.tugraz.at}

\author{Efthymios Sofos}
\address{
Mathematisch Instituut Leiden\\
Universiteit Leiden\\
Snellius building, Niels Bohrweg 1, 2333 CA Leiden, Netherlands\\}
\email{e.sofos@math.leidenuniv.nl}

\begin{abstract} 
Estimating averages of 
Dirichlet convolutions
$1 \ast \chi$,
for some real Dirichlet character $\chi$
of fixed modulus,
over the sparse set of
values of binary
forms
defined over $\Z$
has been the focus of extensive investigations
in recent years,
with spectacular applications to Manin's conjecture for Ch\^atelet surfaces.
We introduce a far-reaching generalization of this problem, 
in particular replacing $\chi$ by
Jacobi symbols with both arguments having varying
size,
possibly tending to infinity.
The main results
of this paper  
provide asymptotic
estimates and lower bounds
of the expected order of magnitude
for the corresponding averages.
All of this is performed
over arbitrary
number fields
by adapting
a technique of Daniel specific to $1\ast 1$.
This is the first time that divisor sums 
over values of binary forms
are asymptotically evaluated 
over any 
number field other than $\Q$.
Our work
is
a 
key
step
in
the proof,
given in subsequent work,
of the lower bound 
predicted by
Manin's
conjecture
for all del Pezzo 
surfaces
over all number fields,
under
mild
assumptions on the Picard
number.
\end{abstract}


\maketitle

\setcounter{tocdepth}{1}

\tableofcontents

\section{Introduction}
\label{intro}
Our aim 
in this paper 
is to 
study
averages of
arithmetic functions
that 
generalise
the divisor function over values of binary forms, defined over
arbitrary number fields.
\subsection{Divisor sums.}
\label{sec:divsum}Estimating averages of
arithmetic functions
is among the primary objects of
analytic number theory and its
applications
to surrounding areas.
Owing to their connection
with
$L$-functions,
two of the most
studied examples
are the divisor 
and the representation function
of sums of two integer squares,
respectively 
given
by 
\[
\tau(n):=
\sum_{\substack{d\in \N \\ d|n}}
\
1
\
\
\text{and}
\
\
r(n):=
4
\!
\sum_{\substack{d\in \N \\ d \ \text{odd} \\ d|n}}
\l(
\frac{-1}{d}
\r)
,\]
where $(\frac{-1}{\cdot})$
denotes the Jacobi symbol,
see for example~\cite[Chapter XII]{titch}.
It is possible to obtain 
level of distribution results,
a problem first studied by Selberg
and Linnik; 
research on this problem
is currently active
due to advances in estimating sums of 
trace functions over finite fields,
see for example~\cite{levelof},
where the ternary divisor function is studied.

Asymptotically estimating 
the average of these
functions
over the sparse set of 
values of
general
integer 
polynomials in a single variable
is naturally harder.
It is only the case of degree $1$ and $2$
polynomials
that has been settled,
see the work of Hooley~\cite{quadra} and of
Duke, Friedlander and Iwaniec~\cite{invention}.
The closely related problem 
regarding integer binary forms
was  
studied later. 
Let us introduce some notation 
to help us
describe previous work on this area. 
For a positive integer
$n$ and each 
$1\leq i \leq n$,
let 
$F_i \in \Z[s,t]$
be forms,
coprime in pairs,
and for any 
constants 
$c_i \in \{1,-1\}$
set 
$
\mathfrak{C}
=
\Big\{(F_i,c_i),
i=1,\ldots,n
\Big\}
$
and
\begin{equation}
\label{eq:gendivsumT}
D(\mathfrak{C};X)
:=
\sum_{\substack{
(s,t) \in (\Z\cap[-X,X])^2 
\\F_i(s,t) \neq 0
}}
\
\
\prod_{i=1}^n
\
\bigg(
\hspace{-0,1cm}
\sum_{\substack{d_i\in \N 
\\ d_i \ \text{odd}
\\ d_i |F_i(s,t)}}
\hspace{-0,2cm}
\l(
\frac{c_i}{d_i}
\r)
\bigg)
,\end{equation}
where the restriction to odd $d_i$ is present only when $c_i=-1$.
The case of degree $3$ was first studied
by Greaves~\cite{greaves},
who obtained an asymptotic for
$D(\mathfrak{C};X)$
when 
$\mathfrak{C}=\{(F,1)\}$
and $F$ is any irreducible form with $\deg(F)=3$
via the use of  
exponential sums.

Extending this result to higher degrees was considered intractable for a long time
until the highly influential work of Daniel~\cite{stephan},
who employed
geometry
of
numbers
to treat the case 
$\mathfrak{C}=\{(F,1)\}$
for any irreducible form $F$ with $\deg(F)=4$.
Developing this approach to allow
negative
$c_i$, 
Heath-Brown~\cite{HB4L}
later tackled the
case
where
$n=4$,
each $c_i$ is $-1$
and all forms $F_i$ are linear.

It was subsequently
realised 
that proving asymptotics 
whenever
$\sum_{i=1}^n\deg(F_i)=4$
would constitute a key step
towards the resolution of Manin's conjecture
for Ch\^{a}telet surfaces over $\Q$.
This is a conjecture in arithmetic geometry
and
regards 
counting rational points of bounded height
on Fano varieties defined over arbitrary number fields;
it was introduced by Manin and his collaborators~\cite{fmt}
in $1989$
and has subsequently
given rise to
a long standing  
research program
that still continues.
Thus, Browning and de la Bret{\`e}che 
reworked later
the case
$\mathfrak{C}=\{(L_i,-1):1\leq i \leq 4\}$,
where each form
$L_i$ is linear in~\cite{compoloi},
the case $\mathfrak{C}=\{(C,-1),(L,-1)\}$,
where $\deg(C)=3$,
$\deg(L)=1$ in~\cite{cubelinear},
and recently Destagnol settled the case 
$\mathfrak{C}=\{(Q,-1),(L_1,-1),(L_2,-1)\}$
with $\deg(Q)=2$, $\deg(L_i)=1$
in~\cite{destagnol}.
In addition,
Browning and de la Bret{\`e}che 
treated the case 
$\mathfrak{C}=\{(Q,1),(L_1,1),(L_2,1)\}$
with $\deg(Q)=2$, $\deg(L_i)=1$ in~\cite{LLQ};
this investigation
formed a significant part in their
proof of Manin's conjecture 
for a smooth quartic del Pezzo surface
for a first time~\cite{dp4}. 
The remaining cases in the divisor sum problem 
with
$\sum_{i=1}^n\deg(F_i)=4$
require
a further development of Daniel's approach,
one that necessitates
the use of
a generalisation of Hooley's delta function~\cite{hooldel}.
This was achieved independently
by
Br\"{u}dern~\cite{bruedern} and 
de la Bret{\`e}che
with Tenenbaum~\cite{MR2927803},
enabling the settling of the 
cases
$\mathfrak{C}=\{(F_1,-1)\}$
and 
$\mathfrak{C}=\{(F_2,-1),(F_3,-1)\}$,
where the forms satisfy 
$\deg(F_1)=4$ and $\deg(F_2)=\deg(F_3)=2$
in~\cite{MR3103132}.

It should be remarked that 
each
work following Daniel
came into fruition only for integer forms $F_i$
fulfilling
a
list 
of extra assumptions regarding the
small prime divisors and the sign of 
the integers
$F_i(s,t)$
as $(s,t)$ ranges through certain regions in $\R^2$.
It will be crucial for our work that 
Daniel's approach
is able of providing
a
polynomial saving in the error term
if
$\sum_{i=1}^n \deg(F_i)=3$
but not when 
$\sum_{i=1}^n \deg(F_i)=4$,
while it has never been extended to any case 
with
$\sum_{i=1}^n \deg(F_i)>4$.

Lastly,
the spectacular work of
Matthiesen~\cite{lilianplms},~\cite{lilacta} and~\cite{mat13},
using tools from
additive combinatorics,
tackled 
all cases where 
$\sum_{i=1}^n \deg(F_i)$
can be arbitrarily large
under the restriction that 
each $F_i$ is linear.
Naturally, this approach does not yield an explicit error term.
\subsection{Generalised divisor sums.}
\label{sec:connection}
In our forthcoming joint work~\cite{fibpub}
with 
Loughran, 
we study Manin's conjecture in dimension 2.
As a
special corollary
we obtain
the lower bound predicted by Manin for all del Pezzo surfaces over all number fields, only under 
mild
assumptions regarding the Picard 
number.
For del Pezzo surfaces of degree $1$ in particular, 
tight lower bounds were not known before,
not even in special cases.
The underlying strategy 
is to 
use
algebro-geometric arguments
to
translate the 
problem into one of estimating 
averages that are a vast
generalisation of the ones appearing in~\eqref{eq:gendivsumT}. 
The success of this strategy
therefore 
relies heavily
on
a very general conjecture
concerning the growth order
of
our
divisor sums;
its precise statement is recorded
in Conjecture~\ref{conj:H}.
In this paper
we prove it 
in all cases
that we need  
for
our
applications to Manin's conjecture, see Theorem~\ref{thm:main 1}. 
In the very
special case
that the base
field is $\Q$,
dealing with a del Pezzo surface of degree $1\leq d \leq 5$
gives birth to
averages of the
rough shape
\begin{equation}
\label{eq:gendivsum}
\sum_{\substack{
(s,t) \in (\Z\cap[-X,X])^2 
\\F_i(s,t) \neq 0
\\(s,t)\equiv (\sigma,\tau) \bmod q
}}
\
\
\prod_{i=1}^n
\
\
h\!\l(F_i(s,t)\r)
\
\bigg(
\hspace{-0,1cm}
\sum_{\substack{d_i\in \N 
\\ d_i \ \text{odd}
\\ d_i |F_i(s,t)}}
\hspace{-0,2cm}
\l(
\frac{G_i(s,t)}{d_i}
\r)
\bigg)
,\end{equation}
where 
$\sigma,\tau,q$
are positive integers,
$h$ is a 
``small''
arithmetic function,
each
$F_i, G_i$
is an integer binary form
with
$\deg(G_i)$ 
divisible by $2$,
all forms
$F_i$ irreducible
and satisfying
\[
\sum_{i=1}^n
\deg(F_i)
=8-d
,\]
which is an integer in the range $\{3,\ldots,7\}$.
Our assumption on $h$
is that it can be written as
$h=1\ast f$, where 
$\ast$ denotes the Dirichlet convolution
and $f$
is a  
multiplicative function on $\N$
that satisfies 
$f(m)=O(\frac{1}{m})$
for $m \in \N$. 
We shall 
call 
a sum as 
in~\eqref{eq:gendivsum}
a \textit{generalised divisor sum}.
This is because $G_i$ are not constants and hence 
the terms are no more a product of multiplicative functions on $\N$
restricted at values of binary forms. 
A further new trait lies in the fact that 
a
\textit{level of distribution}
result 
is required with respect to the modulus $q$,
such a result has not
appeared previously 
for divisor sums over values of polynomials or forms.
In particular,
we shall be able to handle the case
$h(n)=1$ for all $n \in \N$,
thus our results are a true generalisation of previous work
and not a different problem.

A supplementary
aspect of our work is that we estimate asymptotically, for the first time,
divisor sums over values of binary forms in
arbitrary number fields, see
Theorem~\ref{thm:main 2}. 
Thus, 
one 
of the 
central
innovations in our work
lies in
revealing how to extend Daniel's
approach
to this setting.
We shall
rely on a
lattice point counting
theorem of Barroero and Widmer \cite{arXiv:1210.5943}, based on the framework of o-minimal structures.
It is important to note here that
the essence of Daniel's approach lies in taking advantage of the,
possibly large on average,
size of the first successive minima to produce 
a sufficiently small
error term.
Directly adapting
this approach 
to number fields
yields an error term 
whose order
supersedes
the main term;
this would  
preclude the proof of both  
Theorems~\ref{thm:main 1} and~\ref{thm:main 2}.
We shall introduce an 
artifice that
overcomes this difficulty, 
namely we shall
modify Daniel's method by
taking into account not only the first, but also higher successive minima of the lattice.
 
Let us finally state that 
it is not clear what is the expected growth order   
for generalised divisor sums.
The  
r\^{o}le
of
Conjecture~\ref{conj:H}
is to provide an answer
in terms of 
various number fields generated by roots of $F_i(s,1)$.
It is important to note
that
our conjecture will turn out to be 
in agreement with the growth order predicted by Manin's conjecture for surfaces;
this will be revealed in~\cite{fibpub}.

\subsection{Statement of our set-up.}
\label{sec:set-up} 
Throughout  this paper, $K$ will be a number field
of degree $m=[K:\Q]$, whose ring of integers is denoted by 
$\c{O}_K$.
By $\ppp$
and
$\ppp_i$ we always denote non-zero 
prime ideals
of 
$\c{O}_K$
and 
$v_\ppp$
is the $\ppp$-adic exponential evaluation.
\subsubsection{\textbf{Systems of binary forms.}}
\label{sec:systems binary} 
We consider
finite sets 
of pairs of
binary forms
\[
\mathfrak{F}
=
\Big\{(F_i,G_i),
i=1,\ldots,n
\Big\}
,\]
where each $F_i,G_i \in \c{O}_K[s,t]$
is
such that 
$F_i$ is irreducible
and does not divide 
$G_i$ in $K[s,t]$.
Moreover,
we assume that all $F_i$
are coprime over $K$
in pairs
and that each $\deg(G_i)$ is even.

We next define the \textit{rank}
of
$\mathfrak{F}$,
which will be an invariant of $\mathfrak{F}$
that will characterize
the growth order in Conjecture~\ref{conj:H}.
If $F_i$ is proportional to $t$, we denote $\vt_i:=(1,0)$. 
Otherwise, letting $\overline{K}$ be a fixed algebraic closure of $K$, we set $\theta_i \in \overline{K}$ to be a fixed root of $F_i(x,1)$, and $\vt_i := (\theta_i,1)$. Let $K(\vt_i)$ be the subfield of $\overline{K}$ generated by $K$ and the coordinates of $\vt_i$.
We define the {\bf rank} of $\mathfrak{F}$ to be the cardinality
\begin{equation*}
  \rho(\mathfrak{F}) := \card
\big\{1\le i \le n : G_i(\vt_i)\in K(\vt_i)^{\times 2}
\big\}
,
\end{equation*}
where, for any field $k$, we denote the set of its non-zero squares by $k^{\times 2}$.

\subsubsection{\textbf{The group $\c{U}_K$.}}
\label{sec:groupp}
The terms involving the 
function $h$
in~\eqref{eq:gendivsum}
have the r\^{o}le of insignificant modifications. 
We proceed to introduce them precisely.
Letting $\ideals$ denote the monoid of non-zero integral ideals
of $\OO_K$,
$\norm\aaa$ be the absolute norm of $\aaa \in \ideals$
and
$\mu_K$ the M\"{o}bius function on $\ideals$
allows us to introduce the set of functions 
\[
\c{Z}_K:=
\left\{
f:\ideals \to (-1,\infty)
:
\begin{array}{l}
f \ \text{multiplicative}, \\
f(\ppp)\ll_f \frac{1}{\norm\ppp}
\
\text{for all} \
\ppp,
\\
f(\aaa)=0 \text{ if }\mu_K(\aaa)=0
\end{array}
\right\}
.\]
For each  $f \in \c{Z}_K$, we subsequently
define
another
function
$\one_f:\ideals \to (0,\infty)$
given by
\[
\one_f(\aaa):=
\prod_{\ppp|\aaa}
(1+f(\ppp))=(1*f)(\aaa)
.\]
This then allows us to form 
the following set of
positive multiplicative
 functions on $\ideals$,
\begin{equation}
  \label{eq:class 1}
\c{U}_K:=
\big\{
\one_f: 
f \in \c{Z}_K
\big\}
.\end{equation}
The growth condition placed on $f$
indicates that $\one_f$
behaves on average like a constant function.
Note that 
for all $f \in \c{Z}_K$ and $\epsilon>0$
we have 
\begin{equation}\label{doobop}
\one_f(\aaa)\ll_{f,\epsilon} \norm\aaa^\epsilon
,\end{equation}
and moreover, that the set $\c{U}_K$
forms a group under pointwise multiplication.
This will be used often
with the aim of
simplifying the exposition, 
for example 
via replacing terms like 
$\one_{f_1} \one_{f_2}$ or $1/\one_{f_3}$,
where $f_i \in \c{Z}_K$,
by
$\one_f$
for some $f \in \c{Z}_K$.
\subsubsection{\textbf{$\bf{\mathfrak{F}}$\textbf{-admissibility}}}
\label{sec:admissible}
As usual, we shall identify all
completions $K_v$ at
archimedean places $v$ with $\R$ or $\mathbb{C}$. We shall thus let
$
  K_\infty := K\otimes_\Q\R = \prod_{v\mid\infty}K_v,
$
which we identify with $\R^\dg$ via 
$\C\cong \R^2$. In addition, we shall denote by $\DD$ 
a set of the form
$\DD=\prod_{v|\infty} \DD_v$,
where  $\DD_v\subseteq K_v^2$ is a compact ball of positive radius.  
Fixing an integral ideal $\classrep\in\ideals$, we shall consider \emph{$\classrep$-primitive} 
points $(s,t)\in\OO_K^2$, by which we mean that $s\OO_K + t\OO_K = \classrep$.
For an ideal $\www$ of $\OO_K$ divisible by $2\classrep$, and $\aaa\in\ideals$, we define the ideal
\begin{equation}\label{flats}
  \aaa^\flat := \prod_{\substack{\ppp\nmid \www}}\ppp^{v_\ppp(\aaa)},
\end{equation}
and for $a\in\OO_K\smallsetminus\{0\}$, we
let $a^\flat := (a\OO_K)^\flat$. 
Keep in mind that this notion depends on $\www$. 
Let $\sigma,\tau\in\OO_K$ be
such that $\sigma\OO_K + \tau\OO_K = \classrep$.
The symbol $\c{P}$ will
refer exclusively throughout this paper to triplets of the form
\[\c{P}=(\c{D},(\sigma,\tau),\www),\]
where
$\c{D},(\sigma,\tau),\www$ are as above. Given any system of forms $\mathfrak{F}$
as 
in \S \ref{sec:systems binary}, 
a triplet $\c{P}$ and a parameter
$X\geq 1$, we let
\[
  M^*(\c{P},X):=
\{(s,t)\in\classrep^2\cap X^{1/m}\DD \where (s,t)\equiv (\sigma,\tau)\bmod \www
\text{, }s\OO_K + t\OO_K = \classrep\}
\]
and
\[
  M^*(\c{P},\infty) := \bigcup_{X\geq 1}M^*(\c{P},X).
\]
We shall say that $\c{P}$ is
$\bf{\mathfrak{F}}$-\textbf{admissible}
if each of the following conditions~\eqref{eq:f not zero}--\eqref{eq:qr is one}
holds:
\begin{equation}
\label{eq:f not zero}
F_i(\sigma,\tau)\neq 0 \
\text{ for all } 1\leq i \leq n,
\end{equation} 
and
whenever $(s,t)\in M^*(\c{P},\infty)$
we have
\begin{equation}
\label{eq:D not zero}
F_i(s,t)\neq 0 \ \text{ for all } 1\leq i \leq n,
\end{equation}
as well as 
\begin{equation}
\label{eq:qr is one}
\qr{G_i(s,t)}{F_i(s,t)^\flat} = 1
\
\text{ for all } 1\leq i \leq n
.\end{equation}
In the last condition, we used the \emph{Jacobi symbol} for $K$, which is defined as follows: for $a\in\OO_K$ and a non-zero ideal $\bbb = \ppp_1^{e_1}\cdots \ppp_l^{e_l}$,
with distinct prime ideals $\ppp_i$,
none of which lies above $2$, we let
\begin{equation*}
  \qr{a}{\bbb} := \prod_{i=1}^l\qr{a}{\ppp_i}^{e_i},
\end{equation*}
where $\qrp{a}$ is the Legendre
quadratic residue symbol for $K$.
\subsection{\textbf{Lower bound conjecture for generalised divisor sums}}
\label{sec:gensum}
For any $\mathfrak{F}$
as in \S \ref{sec:systems binary}, 
any
function $f \in \c{Z}_K$ and any
triplet $\c{P}$,
we define the
function $r:M^*(\c{P},\infty)\to [0,\infty)$ by
\[
r(s,t) = r(\mathfrak{F},f,\c{P};s,t)
:=
\prod_{i=1}^{n}
\one_{f}(F_i(s,t)^\flat)
\l(\sum_{\ddd_i|F_i(s,t)^{\flat}}\qr{G_i(s,t)}{\ddd_i}\r)
.\]
We are now in the position to 
introduce
\textbf{generalised divisor sums}
as averages of the form
\begin{equation*}
D(\mathfrak{F},f,\c{P};X)
:=
\sum_{\substack{(s,t) \in M^*(\c{P},X)}}
r(\mathfrak{F},f,\c{P};s,t).
\end{equation*}
The special case of the following claim
corresponding to each $G_i$ being constant and $K=\Q$
ought to be familiar, at least among experts,
but has not yet
appeared in text.
\begin{conjecture}[Lower bound conjecture for divisor sums]
\label{conj:H}
Let $K$ be a number field, fix $\classrep\in\ideals$, let $f\in\c Z_K$, and let $\mathfrak{F}$ be a system of forms
as in \S \ref{sec:systems binary}. 
Then there exists a finite set 
$
S_{\bad}
=
S_{\bad}({\mathfrak{F},f,\classrep})$
of prime ideals in $\c{O}_K$,
such that for all
$\mathfrak{F}$-admissible triplets $\c{P}$ with $\www$ being 
divisible by each $\ppp \in S_{\text{bad}}$,
we have
\[
D(\mathfrak{F},f,\c{P};X)
\gg X^2\left(\log X \right)^{\rho(\mathfrak{F})}, \text{ as }X \to \infty
.\]
The implicit constant may depend on every parameter except $X$.
\end{conjecture} 
It should be stated that
the appearance of $G_i,f$ and $\c{P}$
in
Conjecture~\ref{conj:H}, as well as the consideration of arbitrary number fields,
are absolutely necessary for our
applications to Manin's conjecture in~\cite{fibpub}. 
The presence of
the set of bad primes 
$S_{\text{bad}}$
can be avoided;
it is only included here to minimise the technical details in the present work.

We next supply
heuristical 
evidence supporting 
that 
Conjecture~\ref{conj:H} 
does in fact 
provide the true order of magnitude
of
$D(\mathfrak{F},f,\c{P};X)$.
Firstly,
there are
about
$X^2$ summands
and 
each term
$\one_{f}(F_i(s,t)^\flat)$
behaves as a constant on average,
since our 
conditions on $\mathfrak{F}$
suggest that the integral ideals $F_i(s,t)^\flat$
behave randomly.
Secondly,
as we shall see in Lemma~\ref{lem:jacobi trivial},
if the index $i$ contributes towards the rank
$\rho(\mathfrak{F})$
then the Jacobi symbols
$\qr{G_i(s,t)}{\ddd_i}$
assume the value $1$, while
in the opposite case
they take
both values
$1$ and $-1$ 
with equal probability.
Consequently,
in the former case
the sum over 
$\ddd_i|F_i(s,t)^{\flat}$
will resemble the divisor function in $\ideals$,
thus contributing a logarithm,
while
in the latter case
it will be approximated by a 
constant on average owing to the cancellation of the Jacobi symbols.
A subtle point 
here is that if one
does not impose condition~\eqref{eq:qr is one}
then 
the implied constant in the lower bound
could vanish, so the restriction to admissible triplets is necessary.
Furthermore,
each
work referenced in~\S\ref{sec:divsum}
is in agreement with 
Conjecture~\ref{conj:H}
when  
$K=\Q$
and
$G_i=\pm 1$.
Lastly,
the work of
de la Bret{\`e}che
and Browning~\cite{upperdlbtb}
can be used to provide
a matching upper bound over $\Q$ 
whenever 
each $G_i$ is constant.

The main purpose of this paper is to prove
Conjecture \ref{conj:H} 
under a condition regarding only
the
{\bf complexity} of $\mathfrak{F}$,
which we define by
\begin{equation*}
  c(\mathfrak{F}):=
\hspace{-0,3cm}
\sum_{\substack{1\leq i\leq n\\G_i(\vt_i)\notin K(\vt_i)^{\times 2}}}
\hspace{-0,3cm}
\deg F_i  
,\end{equation*}
but without a restriction on
the value of 
$\sum_{i=1}^n\deg(F_i)$ or the factorisation type of $\prod_{i=1}^n F_i$.
\begin{theorem}
\label{thm:main 1}
Conjecture~\ref{conj:H} holds for all 
$K$,
$\classrep$, 
$f$
and
systems of forms
$\mathfrak{F}$ with $c(\mathfrak{F})\leq 3$.
\end{theorem}  
Theorem \ref{thm:main 1} will be reduced to Theorem \ref{thm:main 2}, 
whose statement is 
given in \S\ref{sec:results}. 
\begin{remark}
As an immediate consequence of~\cite[Theorem 1.6]{fibpub}, 
we will see that Conjecture~\ref{conj:H} implies Zariski
density of rational points on conic bundle surfaces over number fields, 
under the necessary assumption that
there is a rational point on a smooth fibre. 
This well-known problem is currently open in most cases, see the recent work of Koll\'{a}r and Mella~\cite{mella}.
\end{remark}

\subsection{Skeleton of the paper and further results}
\label{sec:results} The preliminary
parts,
~\S\ref{s:latpoint}
and~\S\ref{s:artin},
respectively, 
provide 
general
counting results,
that are 
not limited to our applications,
for points of certain lattices
and
averaging results 
concerning 
coefficients of 
Artin $L$-functions.

The reduction of 
Theorem~\ref{thm:main 1}
to
Theorem~\ref{thm:main 2}
below
will take place 
in \S\ref{sec:reduction},
while the proof of the latter theorem will be
given
in~\S\ref{sec:divisor sum asymptotics}.
It provides asymptotics in cases where $\sum_{i=1}^n\deg F_i\leq 3$ and $G_i(\vt_i)\notin K(\vt_i)^{\times 2}$ for all $i$, under some further assumptions.  

It is worth following the strategy laid out in our proof of Theorem~\ref{thm:main 2} to
show that,
for
any positive integers 
$\sigma,\tau,d$
and 
fixed irreducible binary forms $F_i$ with $\sum_{i=1}^n \deg(F_i)\leq 3$,
an
asymptotic estimate with a power saving in terms of $X$ and 
a polynomial dependence on $d$
in the error term
holds
for the analogue of the classical divisor
sums 
\[
\sum_{\substack{(s,t) \in (\Z\cap [-X,X])^2
\\
F_i(s,t)\neq 0
\\(s,t)\equiv (\sigma,\tau) \bmod d
}} 
\
\prod_{i=1}^n
\
\bigg(
\sum_{\substack{d_i\in \N 
\\ d_i|F_i(s,t)}} 
1
\bigg)
\]
over any number field.
We refrain from this task
in the present work
to
shorten the exposition.

We proceed by providing the statement of our second theorem.
We say that an $\mathfrak{F}$-admissible triplet $\c{P}=(\c{D},(\sigma,\tau),\www)$ is {\bf strongly $\mathfrak{F}$-admissible}, if,
in addition,
for all $1\leq i\leq n$
and
$(s,t)\in M^*(\c{P},\infty)$
one has
\begin{equation}
  \label{eq:same power}
 F_i(\sigma,\tau)\nequiv 0 \bmod \www\ \text{ and }\  v_\ppp(F_i(s,t)) = v_\ppp(F_i(\sigma,\tau)) \text{ for all }\ppp\mid \www.
\end{equation} 

\begin{theorem}
\label{thm:main 2}
Let $K$ be a number field, $\classrep\in\ideals$
and $f\in\c{Z}_K$. Let $\mathfrak{F}$ be a system of forms with $\rho(\mathfrak F)=0$ and $c(\mathfrak F)\leq 3$. 
Then there is a non-zero ideal $\www_0$ of $\OO_K$ and 
constants $\beta_1,\beta_2>0$, such that  
the following statement
holds.

For every
strongly $\mathfrak F$-admissible triplet
$\c P = (\c D,(\sigma,\tau), \www)$
fulfilling
$\www_0\mid\www$, there are $\beta_0>0$ and a function $f_0\in\c Z_K$, depending only on $\classrep,f,\mathfrak{F},\c D, \www$, such that for each
 $\ddd\in \ideals$
for which the triplet
$\c P_\ddd := (\c D, (\sigma,\tau), \ddd\www)$
satisfies 
\begin{equation}\label{eq:asympt thm divisor cond}
\prod_{i=1}^{n}
F_i(s,t)\www +  \ddd = \OO_K\quad\text{ for all $(s,t)\in M^*(\c{P}_\ddd,\infty)$},
\end{equation}
the asymptotic 
\[
\sum_{\substack{(s,t)\in M^*(\c P_\ddd,X)}}
r(\mathfrak F,f,\c P; s,t)=
\beta_0
\frac{\one_{f_0}(\ddd)}{\norm\ddd^2}
X^2
+
O(X^{2-\beta_1}\norm\ddd^{\beta_2})
\]
holds with an implied constant independent of $\ddd,\sigma,\tau$ and $X$.
\end{theorem}
This is the first time that
any divisor sum  over values of binary forms
is asymptotically
evaluated over any number field other than $\Q$.
Even over $\Q$, both Theorems~\ref{thm:main 1}
and~\ref{thm:main 2} are 
novel
due to the appearance of the forms $G_i$.
Furthermore, the 
extra condition that $(s,t)$ lies in a progression,
whose modulus is explicitly recorded in the error term,
gives rise to a new
level of distribution
result, since 
an asymptotic 
holds 
when 
$\norm\ddd 
\leq 
X^{\beta}$
for all $0<\beta<\beta_1/\beta_2$.

The power saving in the error term 
of
Theorem~\ref{thm:main 2} is crucial for deducing
Theorem~\ref{thm:main 1} from it, and therefore for
the application to Manin's conjecture. 
Even in the simple case $K=\Q$,
such a strong error term
can
presently
 only be obtained 
under the assumption
$\sum_{i=1}^n \deg(F_i) \leq 3$, 
which 
is the reason for the
restriction placed on
the complexity
$c(\mathfrak{F})$.

As a first step for the proof of Theorem~\ref{thm:main 2},
we use Dirichlet's hyperbola trick
and 
partition the variables in the summation into a 
small number of lattices;
this is exposed
in~\S\ref{finalfinal}.
The next part, residing in~\S\ref{s:minimal application}, 
consists of  
counting
points 
on these lattices;
it 
is
here 
that the main step towards the power saving in the error term 
in Theorem~\ref{thm:main 2} 
takes place.
Finally,
in \S\S 
\ref{s:simplifi}-\ref{s:ending}
we
prove that the 
average of the contribution of 
each lattice alluded to above
gives the main term 
as
stated
in Theorem~\ref{thm:main 2},
this part contains 
the treatment
of volumes 
of 
slightly
awkward regions 
introduced by the consideration of arbitrary number fields.

\subsection*{Acknowledgements}
We are grateful to
Tim Browning
and
Roger
Heath-Brown
for helpful
suggestions relating to
the proof of Theorem~\ref{thm:main 2}. 
The authors would furthermore
like to thank
Daniel Loughran for
useful
discussions
concerning the presentation of our results.
A part of this work was completed while the second author was supported by London's Mathematical Society's \textit{150th Anniversary Postdoctoral Mobility Grant} to visit G\"{o}ttingen University, the hospitality of which is gratefully acknowledged.

\begin{notation}
The set of places of the number field $K$ will be denoted $\places$
and for each $v \in \places$ we shall let 
$\locdegv := [K_v : \Q_w]$,
where $w$ is the place of $\Q$ below $v$. For $a\in\OO_K$, we write $\norm(a) :=\norm(a\OO_K)=\prod_{v\in\archplaces}\absv{a}^\locdegv$ for the absolute value of its norm. For $s\in K_\infty=\prod_{v\in\archplaces}K_v$ and $v\in\archplaces$, we write $s_v\in K_v$ for the projection of $s$ to $K_v$. 
Furthermore,
for any prime ideal $\ppp$
the $\ppp$-adic exponential valuation on ideals (and elements) of $\OO_K$
will 
be
denoted by $v_\ppp$.
As usual,
the resultant of two binary forms $F,G\in\OO_K[s,t]$
will be represented by $\res(F,G)\in\OO_K$,
while Euler's totient function and the divisor function for non-zero ideals of $\OO_K$ will be denoted by $\phi_K$ and $\tau_K$. Lastly, we shall
choose a system of integral representatives $\mathcal{C}=\{\classrep_1, \ldots, \classrep_h\}$ for the ideal class group of $\OO_K$ 
and fix it once and for all.
Unless the contrary is explicitly stated, the implicit constants in Landau's $O$-notation and Vinogradov's $\ll$-notation are allowed to depend on $K,\mathcal{C},\classrep,f,\mathfrak{F}$ and $\c P$ but no other parameters. 
The exact value of a small positive constant $\epsilon$ will be allowed to vary from expression to expression throughout our work.
\end{notation}

\section{Preliminaries}
\label{sec:preliminaries}
\subsection{Lattice point counting}
\label{s:latpoint}
For any lattice $\Lambda\subset K_\infty^2 = \R^{2m}$, we denote its $i$-th successive minimum (with respect to the unit ball) by $\lambda^{(i)}(\Lambda)$. We write $\vecnorm{\cdot}$ for the Euclidean norm on $\R^{2m}$. For $\aaa, \ddd \in\ideals$ and $\gamma\in\OO_K$, we define the lattice
\begin{equation*}
  \Lambda(\aaa,\ddd,\gamma) := \{(s,t)\in \aaa^2 \where s\equiv \gamma t \bmod \ddd\}.
\end{equation*}
It has determinant proportional to $\norm(\aaa^2\ddd(\aaa+\ddd)^{-1})$, and we write $\lambda^{(i)}(\aaa,\ddd,\gamma):=\lambda^{(i)}(\Lambda(\aaa,\ddd,\gamma))$ for its $i$-th successive minimum. Recall that $\mathcal{C} = \{\classrep_1, \ldots, \classrep_h\}$ is a fixed system of integral representatives of the class group of $K$. Let us prove some facts about the minima $\lambda^{(i)}(\aaa,\ddd,\gamma)$.

\begin{lemma}\label{lem:lattice}
Let $\aaa, \ddd \in \ideals$, $\gamma\in\OO_K$ and $1\leq i\leq 2m$.
  \begin{enumerate}
  \item[$(1)$] Whenever $[\aaa] = [\classrep_q]$ for $1\leq q \leq h$,
we have
    \begin{equation*}
\norm\aaa^{1/m}\lambda^{(i)}(\classrep_q,\classrep_q\ddd(\aaa+\ddd)^{-1},\gamma)\ll\lambda^{(i)}(\aaa,\ddd,\gamma)\ll\norm\aaa^{1/m}\lambda^{(i)}(\classrep_q,\classrep_q\ddd(\aaa+\ddd)^{-1},\gamma).
    \end{equation*}
  \item[$(2)$] For any non-zero ideal $\bbb$ of $\OO_K$, the following estimate holds,
    \begin{equation*}
\lambda^{(i)}(\aaa,\ddd,\gamma)\ll\lambda^{(i)}(\aaa,\bbb\ddd,\gamma)\ll\norm(\bbb)^{1/m}\lambda^{(i)}(\aaa,\ddd,\gamma).
\end{equation*}
  \item[$(3)$] We have $\lambda^{(i)}(\aaa,\ddd,\gamma)\ll \norm(\aaa^2\ddd(\aaa+\ddd)^{-1})^{1/(2m-i+1)}$.
  \end{enumerate}
\end{lemma}

\begin{proof}
 Let $ a\in K\smallsetminus\{0\}$ such that $\aaa = a\classrep_q$. Then the
elements $(s,t)\in\aaa^2$ with $s\equiv\gamma t\bmod \ddd$ are exactly those of the form $(s,t) = a(s_1,t_1)$, with
$
  (s_1,t_1)\in \Lambda(\classrep_q,\classrep_q\ddd(\aaa+\ddd)^{-1},\gamma)=:\Lambda'. 
$
By Dirichlet's unit theorem, we can choose our generator $a$ to satisfy
$
\absv{a}\ll\norm\aaa^{1/m}\ll\absv{a}
$
for all $v\in\archplaces$. Then, for any $(s_1,t_1)\in\Lambda'$ we have 
\begin{equation*}
  \norm\aaa^{1/m}\vecnorm{(s_1,t_1)}\ll \vecnorm{a(s_1,t_1)}\ll \norm\aaa^{1/m}\vecnorm{(s_1,t_1)},
\end{equation*}
which
shows claim
$(1)$.
The first inequality of (2) is clear. For the remaining one, let $b \in \bbb$ such that $\absv{b}\ll \norm\bbb^{1/m} \ll \absv{b}$ for all $v\in\archplaces$
and let $(s,t) \in \Lambda(\aaa,\ddd,\gamma)$. This implies that
$(bs,bt) \in \Lambda(\aaa,\bbb\ddd,\gamma)$ and
$  \vecnorm{(bs,bt)}\ll \norm\bbb^{1/m} \vecnorm{(s,t)}
$.
Assertion (3) flows directly from Minkowski's second theorem combined with the obvious fact that $\lambda^{(1)}(\aaa,\ddd,\gamma)\gg 1$.
\end{proof}

We use the framework of \cite{arXiv:1210.5943}, built on o-minimality, to count points of $\Lambda(\aaa,\ddd,\gamma)$ in fairly general domains. Assume we are given an o-minimal structure that extends the semialgebraic structure. Let $\mathcal{R} \subset \R^{k+2\dg}$ be a definable family, such that for each $T\in \R^k$ the fibre
\begin{equation*}
\mathcal{R}_T:=\{(s,t)\in\R^{2\dg}\mid (T,s,t)\in \mathcal{R}\}
\end{equation*}
is contained in a ball, not necessarily zero-centered, of radius $\ll X_T^{1/m}$ for some $X_T \geq 1$. 
The first part of Lemma~\ref{lem:lattice}
makes
the following lemma an immediate consequence of \cite[Theorem 1.3]{arXiv:1210.5943}.
\begin{lemma}\label{lem:lattice point counting}
Whenever $[\aaa] = [\classrep_q]$ and $T\in\R^k$,
the quantity
$  \card(\Lambda(\aaa,\ddd,\gamma)\cap \mathcal{R}_T)$
equals
  \begin{equation*}
 \frac{c_K\vol\mathcal{R}_T}{\norm(\aaa^2\ddd(\aaa+\ddd)^{-1})} + O\left(\sum_{j=0}^{2m-1}\frac{X_T^{j/m}}{\norm\aaa^{j/m}\prod_{i=1}^j\lambda^{(i)}(\classrep_q,\classrep_q\ddd(\aaa+\ddd)^{-1},\gamma)}\right),
  \end{equation*}
with an explicit positive constant $c_K$ depending only on $K$. The implicit constant in the error term may depend on $K,\mathcal{R}$, but not on $T,\aaa,\ddd,\gamma$.
\end{lemma}

Still keeping the notation from above, we now fix an ideal $\classrep\in\ideals$ and assume that $\classrep\mid \aaa$ and that $\aaa+\ddd=\OO_K$. Let $\sigma,\tau\in\classrep$ such that $\sigma\OO_K+\tau\OO_K+\aaa = \classrep$ and define a discrete subset of $K_\infty^2=\R^{2\dg}$ by
\begin{equation}\label{eq:def lattice}
  \Lambda^*(\aaa,(\sigma,\tau),\ddd,\gamma) := 
\left\{
(s,t)\in \classrep^2 \where 
\begin{array}{l}
(s,t)\equiv (\sigma,\tau)\bmod \aaa, \\
s\OO_K+t\OO_K = \classrep, \\
s\equiv\gamma t\bmod \ddd
\end{array}
\right\}
.\end{equation}
Moreover, we require now that
each 
$\mathcal{R}_T$ is
contained in a \emph{zero-centered} ball of radius $\ll X_T^{1/m}$.
\begin{lemma}\label{lem:counting with coprime}
  We have
  \begin{align*}
    \card(\Lambda^*(\aaa,(\sigma,\tau),\ddd,\gamma) \cap \mathcal{R}_T) &- \frac{c_K\vol\mathcal{R}_T}{\zeta_K(2)\norm(\ddd\aaa^2)}\prod_{\ppp\mid\aaa\classrep^{-1}}\left(1-\frac{1}{\norm\ppp^2}\right)^{-1}\prod_{\ppp\mid\ddd}\left(1+\frac{1}{\norm\ppp}\right)^{-1}\\ &\ll 
\sum_{j=0}^{m-1}\frac{X_T^{1+j/m}(\log X_T)\tau_K(\ddd)}{\min_{1\leq q\leq h}\{\lambda^{(1)}(\classrep_q,\classrep_q\ddd,\gamma)^m\lambda^{(m+1)}(\classrep_q,\classrep_q\ddd,\gamma)^j\}}
.
  \end{align*}
Here, $\zeta_K$ is the Dedekind zeta function of $K$
and $\tau_K$ is the divisor function on $\ideals$. The implicit constant in the error term
depends on $K,\classrep,\mathcal{R}$, but not on $T,\aaa,\sigma,\tau,\ddd$
or
$\gamma$. 
\end{lemma}

\begin{proof}
After M\"obius inversion 
the quantity under consideration becomes equal to
  \begin{equation*}
    \sum_{\eee\mid\ddd}\sum_{\substack{
\bbb\in\ideals\\
\bbb+\ddd = \eee \\
\bbb+\aaa\classrep^{-1} = \OO_K
}}
\mu(\bbb)\card\{(s,t)\in (\classrep\bbb)^2\cap(\mathcal{R}_T\smallsetminus\{0\})\where (s,t)\equiv(\sigma,\tau)\bmod \aaa, s\equiv\gamma t\bmod \ddd\}.
  \end{equation*}
  Writing $\bbb = \bbb'\eee$, we see that $\bbb' + \ddd = \OO_K$ whenever $\mu(\bbb)\neq 0$,
thus
the sum becomes
\begin{equation*}
\sum_{\eee\mid\ddd}\mu(\eee)\sum_{\substack{\bbb'\in\ideals\\\bbb'+\aaa\classrep^{-1}\ddd=\OO_K}}\mu(\bbb')\card\{(s,t)\in (\classrep\bbb'\eee)^2\cap(\mathcal{R}_T\smallsetminus\{0\})\where (s,t)\equiv(\sigma,\tau)\bmod \aaa, s\equiv\gamma t\bmod \ddd\}.
  \end{equation*}
Since the set counted in the inner summand is contained in $\Lambda(\classrep\bbb'\eee, \ddd, \gamma)\cap(\mathcal{R}_T\smallsetminus\{0\})$, the summand is zero unless $\lambda^{(1)}(\classrep\bbb'\eee,\ddd,\gamma) \ll X_T^{1/m}$. Using Lemma \ref{lem:lattice}, this condition implies that
\begin{equation}\label{eq:class representative bound}
  \norm\bbb'\ll\frac{X_T}{\min_{1\leq q\leq h}\{\lambda^{(1)}(\classrep_q,\classrep_q\ddd,\gamma)\}^{m}\norm\classrep}.
\end{equation}
Let $\tilde{\sigma},\tilde{\tau}$ in $\classrep\bbb'\eee$ such that $(\tilde{\sigma},\tilde{\tau})\equiv (\sigma,\tau)\bmod\aaa$. 
We have
$(\sigma,\tau)\equiv (0,0)\bmod (\classrep\bbb'\eee+\aaa)=\classrep$,
hence,
such $(\tilde{\sigma},\tilde{\tau})$ exist.
The Chinese remainder theorem 
allows us to transform our sum to
\begin{equation*}
\sum_{\eee\mid\ddd}\mu(\eee)\sum_{\substack{
\eqref{eq:class representative bound}
\\
\bbb'\in\ideals
\\
\bbb'+\aaa\classrep^{-1}\ddd= \OO_K
}}\mu(\bbb')\card\{(s,t)\in ((\tilde{\sigma},\tilde{\tau})+(\aaa\bbb'\eee)^2)\cap(\mathcal{R}_T\smallsetminus\{0\})\where s\equiv\gamma t\bmod \ddd\}.
\end{equation*}
Next, we replace $(s,t)$ by $(s_1,t_1) := (s-\tilde{\sigma}, t-\tilde{\tau})$, so that the inner 
cardinality
becomes
\begin{equation*}
  \card\{(s_1,t_1)\in (\aaa\bbb'\eee)^2\cap ((\mathcal{R}_T\smallsetminus\{0\})-(\tilde{\sigma}, \tilde{\tau}))\where s_1+\tilde{\sigma}-\gamma\tilde{\tau}\equiv \gamma t_1\bmod\ddd \}.
\end{equation*}
Since $\tilde{\sigma}-\gamma\tilde{\tau} \equiv 0 \bmod \eee = \aaa\bbb'\eee+\ddd$, we can find $\delta \in \aaa\bbb'\eee$ with $\delta \equiv \tilde{\sigma}-\gamma\tilde{\tau}\bmod\ddd$. The replacement of $s_1$ by $s_2 := s_1 + \delta$ transforms the count to
\begin{align}
  \nonumber&\card\{(s_2,t_1)\in (\aaa\bbb'\eee)^2\cap((\mathcal{R}_T\smallsetminus\{0\})-(\tilde{\sigma},\tilde{\tau}) + (\delta,0))\where s_2\equiv\gamma t_1\bmod \ddd\}\\ =\ &\card( \Lambda(\aaa\bbb'\eee, \ddd, \gamma)\cap((\mathcal{R}_T\smallsetminus\{0\})-(\tilde{\sigma},\tilde{\tau})+(\delta,0))).\label{eq:count 0}
\end{align}
Clearly, we can extend our family $\mathcal{R}$ to a definable family $\widetilde{\c{R}}\subseteq \R^{(k+2\dg)+2\dg}$, whose fibre $\widetilde{\c{R}}_{(T,\sigma,\tau)}$, for $(T,\sigma,\tau)\in\R^{k+2\dg}$, is the translate $\mathcal{R}_T + (\sigma,\tau)$. Lemma \ref{lem:lattice point counting} thus allows us to approximate the quantity in \eqref{eq:count 0} by
\begin{equation}\label{eq:count}
  \frac{c_K\vol\mathcal{R}_T}{\norm(\aaa^2\bbb'^2\eee\ddd)}+O\left(\sum_{j=0}^{2m-1}\frac{X_T^{j/m}}{\norm(\aaa\bbb'\eee)^{j/m}\min_{1\leq q\leq h}\{\prod_{i=1}^j\lambda^{(i)}(\classrep_q,\classrep_q\ddd\eee^{-1},\gamma)\}}\right).
\end{equation}
Summing the main term over $\eee$ and $\bbb'$ gives
\begin{equation*} \frac{c_K\vol\mathcal{R}_T}{\norm(\aaa^2\ddd)}\sum_{\eee\mid\ddd}\frac{\mu_K(\eee)}{\norm\eee}\sum_{\substack{\bbb'\in\ideals\\\bbb'+\aaa\ddd\classrep^{-1} = \OO_K\\\eqref{eq:class representative bound}}}\frac{\mu_K(\bbb')}{\norm\bbb'^2}.
\end{equation*}
The desired main term is obtained by removing
condition~\eqref{eq:class representative bound},
present in the inner sum. This introduces an error of size 
\begin{align*}
  \ll \frac{\vol\mathcal{R}_T}{X_T\norm\ddd}\sum_{\eee\mid\ddd}\norm\classrep\ \min_{1\leq q\leq h}\{\lambda^{(1)}(\classrep_q,\classrep_q\ddd,\gamma)\}^m &\ll \frac{\tau_K(\ddd)\vol\mathcal{R}_T\min_{1\leq i\leq h}\{\lambda^{(1)}(\classrep_q,\classrep_q\ddd,\gamma)\}^m\norm\classrep}{X_T\norm\ddd}\\ &\ll \frac{X_T\tau_K(\ddd)\norm\classrep}{\min_{1\leq i\leq h}\{\lambda^{(1)}(\classrep_q,\classrep_q\ddd,\gamma)\}^m}.
\end{align*}
Summing the summand for $j$ in the error term of \eqref{eq:count} over $\eee$ and $\bbb'$ gives a total error
\begin{equation}\label{eq:lattice error term}
  \ll X_T^{j/m}\sum_{\eee\mid\ddd}\frac{1}{\norm\eee^{j/m}\min_{1\leq q\leq h}\{\prod_{i=1}^j\lambda^{(i)}(\classrep_q,\classrep_q\ddd\eee^{-1},\gamma)\}}\sum_{\substack{\bbb'\in\ideals\\\eqref{eq:class representative bound}}}\frac{1}{\norm\bbb'^{j/m}}
\end{equation}
and 
\begin{align*}
\sum_{\substack{\bbb'\in\ideals\\\eqref{eq:class representative bound}}}\frac{1}{\norm\bbb'^{j/m}}& \ll \left(\frac{X_T}{\min_{1\leq q\leq h}\{\lambda^{(1)}(\classrep_q,\classrep_q\ddd,\gamma)\}^m}\right)^{\max\{0,1-j/m\}}(\log X_T).
\end{align*}
Observe, moreover, that $\norm\eee^{1/m}\lambda^{(i)}(\classrep_q,\classrep_q\ddd\eee^{-1},\gamma)\gg \lambda^{(i)}(\classrep_q,\classrep_q\ddd,\gamma)$, by Lemma \ref{lem:lattice}. Thus, for $j\geq m$ the expression in \eqref{eq:lattice error term} is
\begin{equation*}
  \ll \frac{X_T^{j/m}(\log X_T)\tau(\ddd)}{\min_{1\leq q\leq h}\{\lambda^{(1)}(\classrep_q,\classrep_q\ddd,\gamma)^{m}\lambda^{(m+1)}(\classrep_q,\classrep_q\ddd,\gamma)^{j-m}\}},
\end{equation*}
which, upon replacing $j$ by $j-m$, 
is
covered by the lemma's error term. For $j<m$, the expression in \eqref{eq:lattice error term} is
at most
$
  \ll  
X_T(\log X_T)\tau(\ddd)
(\min_{1\leq q\leq h}\{\lambda^{(1)}(\classrep_q,\classrep_q\ddd,\gamma)^{m}\})^{-1}
$.
\end{proof}

\subsection{Averages of certain arithmetic functions related to Artin $L$-functions}
\label{s:artin}
We shall 
provide asymptotic estimates
for averages of functions 
that will later appear in
the treatment of the main term in
Theorem~\ref{thm:main 2}.

\begin{lemma}\label{lem:contour integration}
Let $a:\N\to \C$ be an arithmetic function with associated Dirichlet series $A(s) = \sum_{n\in\N}a(n)n^{-s}$. Let $\delta,C > 0$, $\lambda > 2$
and assume that
\begin{align}
&a(n) \leq C n^\delta,\label{eq:contour integration an bound}\\ 
&A(s) \text{ has an analytic continuation to }\Re(s)>1/2,\\
&A(s) \leq C (1+|\Im(s)|)^{1/2}\text{, for } \Re(s)\geq 1-1/\lambda.
\end{align}
Then
\begin{equation*}
  \sum_{n\leq X}a(n) \ll C X^{1-1/(2\lambda)+2\delta}, 
\end{equation*}
for $X\geq 1$, where the implicit constant may depend at most on $\lambda$ and $\delta$.
\end{lemma}

\begin{proof}
  The Dirichlet series defining $A(s)$ converges absolutely for $\Re(s)>1+\delta$, thanks to \eqref{eq:contour integration an bound}. Let $\sigma_0 := 1+2\delta$ and $T:=X^{1/\lambda}$. 
We shall make use of Perron's formula (see for example~\cite[Corollary 5.3]{vaughan060})
to obtain
  \begin{align*}
\sum_{n\leq X}a(n) 
&
- \frac{1}{2\pi i}\int_{\sigma_0-iT}^{\sigma_0+iT}A(s)\frac{X^s}{s}\mathrm{d} s\\
&\ll \sum_{x/2<n<2x}|a(n)|\min\left\{1,\frac{X}{T|X-n|}\right\} + 
\frac{4^{\sigma_0} + X^{\sigma_0}}{T}
\sum_{n\in\N}\frac{|a(n)|}{n^{\sigma_0}}.
  \end{align*}
Replacing the minimum by its second term unless $|X-n|<1$, the first error term becomes
\begin{equation*}
  \ll CX^{\delta}\left(1+\frac{X}{T}\sum_{1\leq m \leq 2X}\frac{1}{m}\right)\ll_\delta CX^{1-1/\lambda+2\delta},
\end{equation*}
while 
the second error term is
$\ll C X^{1-1/\lambda+2\delta}\sum_{n\in\N}n^{-1-\delta}\ll_{\delta}CX^{1-1/\lambda+2\delta}$. 
Shifting the line of integration to the left, we see 
that the main term
equals
\[
  \left(-\int_{1-1/\lambda-iT}^{\sigma_0-iT}+\int_{1-1/\lambda -iT}^{1-1/\lambda+iT} + \int_{1-1/\lambda+iT}^{\sigma_0+iT}\right)A(s)\frac{X^s}{s}\mathrm{d}s.
\]
The first and third integral 
are bounded by
\begin{equation*}
  \ll CT^{-1/2}\int_{u=1-1/\lambda}^{\sigma_0}X^u\mathrm{d}u \ll CT^{-1/2}X^{\sigma_0} = CX^{1-1/(2\lambda)+2\delta}
\end{equation*}
and
the second integral attains a value
\begin{align*}
  &\ll CX^{1-1/\lambda}\int_{t=-T}^T\frac{(1+|t|)^{1/2}}{|1-1/\lambda+it|}\mathrm{d}T \ll CX^{1-1/\lambda}\left(1+\int_{t=1}^Tt^{-1/2}\mathrm{d}t\right)\\ &\ll CX^{1-1/\lambda} T^{1/2}\ll CX^{1-1/(2\lambda)}.
\end{align*}
\end{proof}

\begin{lemma}\label{lem:twisted sum}
  Let $\rho:\ideals\to\C$ be a multiplicative function whose assosiated Dirichlet series is
$D_\rho(s) = \sum_{\aaa\in\ideals}\rho(\aaa)\norm\aaa^{-s}$. Let $\www\in\ideals$, $\lambda>2$, and 
$f\in\mathcal{Z}_K$. Assume that the following conditions hold:
  \begin{align}
    &\rho(\aaa)=0 \text{ unless }\aaa+\www=\OO_K,\\
    &\rho(\ppp^k)\ll_\rho 1 \text{ for all prime ideals } \ppp\nmid\www \text{ and all }k\geq 0,\label{eq:twisted sum cond 2}\\
    &D_\rho(s) \text{ has an analytic continuation to }\Re(s)>1/2,\label{eq:twisted sum cond 3}\\
    &D_\rho(s) \ll_\rho (1+|\Im(s)|)^{1/2} \text{ for }\Re(s)\geq 1-1/\lambda,\label{eq:twisted sum cond 4}\\
&\left|\sum_{k=1}^\infty\frac{\rho(\ppp^k)}{\norm\ppp^{ks}}\right|<\frac{1}{2} \text{ for all prime ideals }\ppp\nmid\www \text{ and }\Re(s)>1/2,\label{eq:twisted sum cond 5}\\
&\left|\one_f(\ppp)\sum_{k=1}^\infty\frac{\rho(\ppp^k)}{\norm\ppp^{ks}}\right|<\frac{1}{2} \text{ for all prime ideals }\ppp\nmid\www \text{ and }\Re(s)>1/2.\label{eq:twisted sum cond 6}
  \end{align}
Then there is $\beta>0$ and $\upgamma\in
\mathcal{Z}_K$, such that, for any $\ccc\in\ideals$ with $\ccc+\www=\OO_K$, we have
\begin{equation*}
  \sum_{\substack{\norm\aaa\leq X\\\aaa+\ccc\www=\OO_K}}\frac{\one_f(\aaa)\rho(\aaa)}{\norm\aaa} = D_\rho(1)\beta\one_\upgamma(\ccc) + O(\norm\ccc^\epsilon X^{-1/(2\lambda)+\epsilon}),
\end{equation*}
for all $\epsilon>0$. The implicit constant is allowed to depend on $\epsilon,\rho,\www,f,\lambda$, but not on $\ccc,X$.
\end{lemma}

\begin{proof}
  For $\ppp\nmid\www$ let 
$\Phi_\ppp(s) := \sum_{k=1}^\infty
\rho(\ppp^k) \norm\ppp^{-ks}$,
  which is bounded in absolute value by $1/2$ whenever $\Re(s)>1/2$, due to \eqref{eq:twisted sum cond 5}. Moreover, condition \eqref{eq:twisted sum cond 2} implies that
  \begin{equation}
    \label{eq:twisted sum phi p bound}
    \Phi_\ppp(s) \ll_\rho \norm\ppp^{-s}\quad \text{ for }\Re(s)>1/2.
  \end{equation}
  Define formally the Dirichlet series
  \begin{align*}
D_\ccc(s) &:= \sum_{\substack{\aaa\in\ideals\\\aaa+\ccc\www=\OO_K}}\frac{\one_f(\aaa)\rho(\aaa)}{\norm\aaa^s}=\prod_{\ppp\nmid\ccc\www}\left(1+\one_f(\ppp)\Phi_\ppp(s)\right),\\
\Psi_\ccc(s) &:= \prod_{\ppp\mid\ccc}\left(1+\one_f(\ppp)\Phi_\ppp(s)\right)^{-1}
  \ \ \ \ \  \ \ \ \ \ \ \text{and}  \\
\Phi(s) \ &:= 
\prod_{\ppp\nmid\www}\frac{1+\one_f(\ppp)\Phi_\ppp(s)}{1+\Phi_\ppp(s)} = \prod_{\ppp\nmid\www}\left(1 + \frac{f(\ppp)\Phi_\ppp(s)}{1+\Phi_\ppp(s)}\right),
  \end{align*}
  to obtain a factorization
  \begin{equation}\label{eq:Dcs factorization}
    D_\ccc(s) = D_\rho(s)\Phi(s)\Psi_\ccc(s).
  \end{equation}
By \eqref{eq:twisted sum phi p bound}, the Euler products for $D_\ccc(s)$ and $D_\rho(s)$ converge absolutely and define holomorphic functions for $\Re(s)>1$,
while~\eqref{eq:twisted sum phi p bound} and~\eqref{eq:twisted sum cond 5}
guarantee that $\Phi(s)$ converges absolutely and defines a holomorphic function on $\Re(s)>1/2$. 
Moreover, \eqref{eq:twisted sum cond 6} ensures that all factors of the finite product $\Psi_\ccc(s)$ are defined and holomorphic for $\Re(s)>1/2$. 
Consequently,
the factorization \eqref{eq:Dcs factorization} holds for $\Re(s)>1$ and, using  \eqref{eq:twisted sum cond 3}, provides an analytic continuation of $D_\ccc(s)$ to $\Re(s)>1/2$. For $\Re(s)\geq 1-1/\lambda$, we obtain by \eqref{eq:twisted sum cond 4} and \eqref{eq:twisted sum cond 5} that
\begin{equation*}
  |D_\ccc(s)|\ll_\rho (1+|\Im(s)|)^{1/2}\left(\prod_{\ppp\mid\ccc}2\right)|\Phi(s)| \ll_{\epsilon,f,\rho,\lambda} \norm\ccc^\epsilon(1+|\Im(s)|)^{1/2}. 
\end{equation*}
Since moreover
$\sum_{\norm\aaa=k}\one_f(\aaa)\rho(\aaa) \ll_{\epsilon,f,\rho} k^{\epsilon}$,
we may apply Lemma \ref{lem:contour integration} to obtain
for any $\epsilon>0$,
\begin{equation*}
  \sum_{\substack{\norm\aaa\leq X\\\aaa+\ccc\www=\OO_K}}\one_f(\aaa)\rho(\aaa) \ll_{\epsilon,f,\rho,\lambda}\norm\ccc^\epsilon X^{1-1/(2\lambda)+\epsilon}
.\end{equation*}
Partial summation reveals that the series defining $D_\ccc(s)$ converges for $s=1$
and 
\begin{equation*}
  \sum_{\substack{\norm\aaa\leq X\\\aaa+\ccc\www=\OO_K}}\frac{\one_f(\aaa)\rho(\aaa)}{\norm\aaa} = D_\rho(1)\Phi(1)\Psi_\ccc(1) + O(\norm\ccc^\epsilon X^{-1/(2\lambda)+\epsilon}).
\end{equation*}
Conditions \eqref{eq:twisted sum cond 5} and \eqref{eq:twisted sum cond 6} show that $\beta := \Phi(1) > 0$. We finish our proof with the observation $\Psi_\ccc(1) = \one_\upgamma(\ccc)$, where
\begin{equation*}
  \upgamma(\ppp) := (1+\one_f(\ppp)\Phi_\ppp(1))^{-1}-1 = \sum_{k=1}^\infty(\one_f(\ppp)\Phi_\ppp(1))^k.
\end{equation*}
In particular, $|\upgamma(\ppp)| < 1$ and $\upgamma(\ppp)\ll \norm\ppp^{-1}$, so 
$\upgamma\in\mathcal{Z}_K$.
\end{proof}

In our proof of Theorem \ref{thm:main 2}, we shall
apply the above result for Dirichlet series $D_\rho(s)$ of the following form. Let $(F,G)$ be a pair of binary forms in $\OO_K[s,t]$, such that $F$ is irreducible in $K[s,t]$, not proportional to $t$,  and does not divide $G$ in $K[s,t]$. We assume furthermore that $G$ is of even degree, and that $G(\theta,1)\notin K(\theta)^{\times 2}$, where $\theta\in\overline{K}$ is a root of 
$F(s,1)$.

Fix $\www\in \ideals$ with $2\mid\www$. We define, for $\aaa\in\ideals$, 
the multiplicative
function $\rho_{(F,G)}(\aaa)$ by
\begin{equation*}
\rho_{(F,G)}(\aaa) := \sum_{\substack{\lambda \bmod{\aaa}\\F(\lambda,1)\equiv 0 \bmod{\aaa}}} \qr{G(\lambda,1)}{\aaa},\quad\text{ if }\aaa+\www=\OO_K,
\end{equation*}
and $\rho_{(F,G)}(\aaa)=0$ otherwise. 
We assume that $\www$ is divisible by enough small prime ideals to ensure that $2\cdot |\rho_{(F,G)}(\ppp)|<\norm\ppp^{1/2}$ for all prime ideals $\ppp$.

\begin{lemma}
\label{lem:artinakos}
 The Dirichlet series
of $\rho_{(F,G)}$,
given by
\begin{equation*}
D_{(F,G)}(s):= \sum_{\aaa \in \ideals}\frac{\rho_{(F,G)}(\aaa)}{\norm\aaa^s}
,\end{equation*}
defines a holomorphic function in $\Re(s)>\frac{1}{2}$ that does not vanish at $s=1$. 
We
furthermore
have
$
|D_{(F,G)}(s)| \ll \l(1+|\Im(s)|\r)^{1/2}
$
in the region $\Re(s) > 1-1/\lambda$, where $\lambda =1+2\dg\deg F$.
\end{lemma}

\begin{proof}
Let $a:=F(1,0)\in \OO_K\smallsetminus\{0\}$. Then $F(s,a t)=a\widehat{F}(s,t)$, where $\widehat{F}(s,1) \in \OO_K[s]$ is monic and irreducible.
Note that the constant
$\widehat{\theta}:=a\theta$
is a root of
$\widehat{F}(s,1)$.
Define the number field $H := K(\theta,\sqrt{G(\widehat{\theta}, a)}) = K(\theta,\sqrt{G(\theta,1)})$,
which clearly fulfills
$[H : K(\theta)] = 2$.

The non-trivial representation of $\text{Gal}(H/K(\theta))$ gives rise to the Artin $L$-function
\begin{equation*}
L(s,\chi)=\prod_{\mathfrak{P}}\left(1-\chi(\mathfrak{P})\norm_{K(\theta)/\Q}\mathfrak{P}^{-s}
\right)^{-1},
\end{equation*}
with the product running over the non-zero prime ideals $\mathfrak{P}$ of $K(\theta)$. The character $\chi(\mathfrak{P})$ is $0$ if $\mathfrak{P}$ is ramified in $H/K(\theta)$ and $1$ or $-1$ according to whether $\mathfrak{P}$ is split or inert in $H/K(\theta)$. This $L$-function is entire and does not vanish at $s=1$. The usual argument about split primes shows that
\begin{equation*}
\prod_{\mathfrak{P}\mid\ppp}(1+\chi(\mathfrak{P})\norm_{K(\theta)/\Q}\mathfrak{P}^{-s}) = 1 + \left(\sum_{\substack{\mathfrak{P}\mid\ppp\\f(\mathfrak{P}/\ppp)=1}}\chi(\mathfrak{P})\right)\norm\ppp^{-s} + O(\norm\ppp^{-2s}),
\end{equation*}
for every prime ideal $\ppp$ of $\OO_K$, where $f(\mathfrak{P}/\ppp)$ is the inertia degree.

In the following considerations, we assume that $\ppp$ is relatively prime to $a$ and to the conductors of the orders $\OO_K[\widehat{\theta}]$ in $K(\theta)$ and $\OO_{K(\theta)}[\sqrt{G(\widehat\theta,a)}]$ in $H$. Then the primes $\mathfrak{P}$ in $K(\theta)$ above $\ppp$ with $f(\mathfrak{P}/\ppp)=1$ are parameterized by the roots $\lambda$ of $\widehat{F}(s,1)$ modulo $\ppp$. If $\mathfrak{P}$ corresponds to the root $\lambda$, then we have an isomorphism
$
  \OO_{K(\theta)}/\mathfrak{P}\to \OO_{K}/\ppp
$
given by $\widehat{\theta}\mapsto \lambda$. 
Consequently,
\begin{equation*}
\chi(\mathfrak{P}) = \qr{G(\widehat{\theta},a)}{\mathfrak{P}} = \qrp{G(\lambda,a)}
\end{equation*}
and in particular,
\begin{equation*}
\sum_{\substack{\mathfrak{P}\mid\ppp\\f(\mathfrak{P}/\ppp)=1}}\chi(\mathfrak{P}) = 
\sum_{\substack{\lambda\bmod\ppp\\\widehat{F}(\lambda,1)
\equiv 0\bmod\ppp}}\qrp{G(\lambda,a)} = 
\sum_{\substack{\lambda\bmod\ppp\\
F(\lambda,1) 
\equiv 0\bmod\ppp}}\qrp{G(a\lambda,a)} 
= \rho_{(F,G)}(\ppp),
\end{equation*}
where we again relied on the fact that $G$ is of even degree. Let $\www_1$ be the product of all the prime ideals excluded above. We have shown that
\begin{equation*}
  L(s,\chi) = g_0(s)\prod_{\ppp\nmid\www_1\www}\left(1+\frac{\rho_{(F,G)}(\ppp)}{\norm\ppp^s}\right) = g_1(s)\prod_{\ppp\nmid\www}\left(1+\frac{\rho_{(F,G)}(\ppp)}{\norm\ppp^s}\right) = g_2(s)D_{(F,G)}(s),
\end{equation*}
where $g_0, g_1, g_2$ are holomorphic functions and have ablosutely convergent Euler products on 
$\Re(s) > 1/2$ that do not vanish there. Hence,
for $\Re(s) > 1/2+\epsilon$.
we have 
$1 \ll_{\epsilon} g_2(s) \ll_{\epsilon} 1$.

Convexity bounds,  
for example \cite[Theorem III.14 A]{MR2135107} with $\eta=1/(2\dg\deg F)$, 
show that
\begin{equation*}
  L(s,\chi)\ll (1+|\Im(s)|)^{1/2} \quad\text{ in }\quad 1-\eta \leq \Re(s) \leq 1+\eta,
\end{equation*}
which extends to the region $1-\eta \leq \Re(s)$ by absolute convergence of $L(s,\chi)$ in $\Re(s)>1$.
\end{proof}
We 
shall need to handle 
averages of
volumes 
of certain regions (see~\eqref{eq:ferume}).
The next version of 
Abel's sum formula
is optimally
tailored for this task. 
\begin{lemma}\label{lem:abel discrete}
  Let $g,\omega:\N\to\C$ be functions, and write $G(u):=\sum_{n\leq u}g(n)$. Let $X\geq 1$, $A,B\geq 0$ with $A+B<1$,  and assume that
  \begin{enumerate}
  \item[(1)] $\omega(n)=0$ for $n\geq X$,
  \item[(2)] there is $Q\geq 0$ such that $|\omega(n)-\omega(n+1)|\leq Qn^{-B}$ holds for all $n\in\N$,
  \item[(3)] there are $\lambda_0\in\C$, $M\geq 0$, such that $|G(n)-\lambda_0| \leq M n^{-A}$ holds for all $n\in\N$.
  \end{enumerate}
Then
\begin{equation*}
  \left|\sum_{n\leq X}g(n)\omega(n) - \lambda_0\omega(1)\right| \leq MQ\left(1+\frac{X^{1-A-B}}{1-A-B}\right).
\end{equation*}
\end{lemma}

\begin{proof}
  Telescoping and using assumption \emph{(1)}, we see that
  \begin{align*}
    \sum_{n\leq X}g(n)\omega(n) &= \sum_{n\leq X}G(n)(\omega(n)-\omega(n+1))\\ &= \lambda_0\sum_{n\leq X}(\omega(n)-\omega(n+1)) + \sum_{n\leq X}(G(n)-\lambda_0)(\omega(n)-\omega(n+1)).
  \end{align*}
 The first summand is equal to $\lambda_0\omega(1)$, and, using assumptions \emph{(2)} and \emph{(3)}, the last sum has absolute value at most
 \begin{equation*}
   MQ\sum_{n\leq X}n^{-A-B}\leq MQ\left(1+\int_{1}^X\frac{\mathrm{d}u}{u^{A+B}}\right)\leq MQ\left(1+\frac{X^{1-A-B}}{1-A-B}\right).
 \end{equation*}
\end{proof}

\section{Proof of Theorem \ref{thm:main 1}}
\label{sec:reduction}

In this section we 
assume the validity of 
Theorem~\ref{thm:main 2}
and we
prove 
Theorem~\ref{thm:main 1} from it. 
The finite
set $S_\bad$ will contain all prime ideals that we want to exclude at various steps of our argument. It will grow during the proof, but it will never depend on anything but $K$, $\classrep$, $\mathfrak{F}$
and
$f$.
In Theorems~\ref{thm:main 1} and~\ref{thm:main 2},
we will always assume that none of the forms $F_i(s,t)$ is proportional to $t$. This can be achieved by a unimodular transformation $\phi_a:K^2\to K^2$, $(s,t)\mapsto (s,as+t)$, for suitable $a\in\OO_K$. This map $\phi_a$ extends to $K_\infty^2\to K_\infty^2$ in an obvious way, transforming $\mathcal{D}$ to $\phi_a(\mathcal{D})$. Clearly, all our hypotheses are still satisfied.

\subsection{Simple reductions} 
\begin{lemma}
\label{lem:thinning}
Let $\c{P}=(\c{D},(\sigma,\tau),\www)$ be an $\mathfrak{F}$-admissible triplet, and $k \in \N$. Then
\begin{equation*}
\c{P}^k:=(\c{D},(\sigma,\tau),\www^k)
\end{equation*}
is also an $\mathfrak{F}$-admissible triplet and
$
D(\mathfrak{F},f,\c{P};X)
\gg_k
D(\mathfrak{F},f,\c{P}^k;X)
$.
\end{lemma}
\begin{proof} 
  Since $\www$ and $\www^k$ have the same prime factors, the ideals $\aaa^\flat$, for $\aaa\in\ideals$, are the same for $\www$ and $\www^k$. Moreover, $M^*(\c{P}^k,X)\subseteq M^*(\c{P},X)$. This shows that, $\c{P}^k$ is admissible, and moreover $r(\mathfrak{F},f,\c{P}; s,t)=r(\mathfrak{F},f,\c{P}^k; s,t)$. The lemma follows immediately, since $r(\mathfrak{F},f,\c{P}; s,t)\geq 0$.
\end{proof}

It is enough to prove Conjecture \ref{conj:H} for all strongly $\mathfrak{F}$-admissible triplets. Indeed, given any $\mathfrak{F}$-admissible triplet $\c{P}=(\c D,(\sigma,\tau),\www)$, we may assume it to be strongly $\mathfrak{F}$-admissible. To this end, we may replace $\www$ by any positive power of itself, thanks to Lemma \ref{lem:thinning}. By \eqref{eq:f not zero}, we can find $k\in \N$, such that $\c P^k$ satisfies \eqref{eq:same power}.

By including in $S_\bad$ enough small prime ideals and replacing $\www$ by a high enough power, we can moreover assume that
\begin{equation}\label{eq:w leading coeffs}
2\classrep\prod_{i}F_i(1,0)\prod_{i\neq j}\res(F_i,F_j)\mid \www.
\end{equation}

\subsection{
Eclipsing the trivial $G_i$}
\begin{lemma}\label{lem:jacobi trivial}
Whenever
$i\in\{1,\ldots,n\}$
is
such that $G_i(\vt_i)\in K(\vt_i)^{\times 2}$,
then
for all $s,t\in \OO_K$ with $s\OO_K + t\OO_K = \classrep$
we have 
  \begin{equation*}
    \sum_{\ddd_i\mid F_i(s,t)^\flat}\qr{G_i(s,t)}{\ddd_i} = \tau_K(F_i(s,t)^\flat).
  \end{equation*}
\end{lemma}

\begin{proof}
  The isomorphism $K[S]/F_i(S,1)\to K(\vt_i)$, $S\mapsto \theta_i$, sends $G_i(S,1)$ to $G_i(\vt_i)$. Hence,
\begin{equation*}
G_i(S,1)=h(S)^2 + c(S)F_i(S,1),
\end{equation*}
with polynomials $h(S),c(S) \in K[S]$, such that $F_i(S,1)\nmid h(S)$. Let $d$ be the maximum of the degrees of $G_i(S,1)$, $h(S)^2$, $c(S)F_i(S,1)$. Re-homogenizing, we obtain
\begin{equation*}
  G_i(S,T)T^{d-\deg G_i} = H(S,T)^2T^{d-2\deg H} + C(S,T)T^{d-\deg C-\deg F_i}F_i(S,T),
\end{equation*}
with forms $H,C\in K[S,T]$. Letting $b\in\OO_K$ such that $bH(S,T)\in\OO_K[S,T]$, we find that  $\res(bH(S,T),F(S,T))\in\OO_K\smallsetminus\{0\}$. After adding to $S_\bad$ all prime ideals that divide $b\res(bH(S,T),F(S,T))$, and all modulo which the form $C$ can not be reduced, we obtain, for all $s,t\in\OO_K$ and all $\ppp\mid F_i(s,t)^\flat$,
\begin{equation*}
  \qr{G_i(s,t)t^{d-\deg G}}{\ppp} = \qr{H(s,t)^2t^{d-2\deg H}}{\ppp}.
\end{equation*}
Using
$s\OO_K + t\OO_K = \classrep$
and 
$\ppp\nmid F_i(1,0)$,
we see that if
$\ppp\mid t$ then $\ppp\mid s$, which shows that
$\ppp\mid\classrep\mid\www$, a contradiction. Hence, $t$ is invertible modulo $\ppp$ and
using that $\deg G$ is even, we derive
\begin{equation*}
\qr{G_i(s,t)}{\ppp} = \qr{H(s,t)}{\ppp}^2 \qr{t}{\ppp}^{\deg G-2\deg H} = \qr{H(s,t)}{\ppp}^2=1
.\end{equation*}
In the last equality, we were allowed to exclude the case $H(s,t)\equiv 0\bmod\ppp$ due to the condition $\ppp\nmid\res(bH(S,T),F_i(S,T))$.
\end{proof}
By possibly reordering the $(F_i,G_i)\in\mathfrak{F}$, we may assume that
\begin{equation*}
G_i(\vt_i)
\begin{cases}
  \in K(\vt_i)^{\times 2} &\text{ for }1\leq i \leq  \rho(\mathfrak{F}),\\ 
  \notin K(\vt_i)^{\times 2} &\text{ for }\rho(\mathfrak{F})+1\leq i\leq n.
\end{cases}
\end{equation*}
We define $f'(\ppp):=0$ if $\ppp\in S_\bad$ and $f'(\ppp):=2f(\ppp)$ otherwise.
Note that
choosing $S_\bad$ large enough ensures that $f'\in \c{Z}_K$. 
All $n$ factors in the definition of $r(s,t)$ are non-negative and for $1\leq i\leq \rho(\mathfrak{F})$
 we  see by Lemma \ref{lem:jacobi trivial}
that
\begin{align*}
  \one_f(F_i(s,t)^\flat)\sum_{\ddd_i\mid F_i(s,t)^\flat}\qr{G_i(s,t)}{\ddd_i} &= \prod_{\ppp\mid F_i(s,t)^\flat}(1+f(\ppp))(v_\ppp(F_i(s,t))+1) 
\\&\geq 
\prod_{\ppp\mid F_i(s,t)^\flat}(1+(1+2f(\ppp)))
= \sum_{\substack{\ddd_i\mid F_i(s,t)\\\ddd_i+\www=\OO_K}}\mu_K^2(\ddd_i)\one_{f'}(\ddd_i).
\end{align*}
If $\rho(\mathfrak{F}) <n$, we let $\mathfrak F' := \{(F_{\rho(\mathfrak F)+1},G_{\rho(\mathfrak F)+1}), \ldots, (F_n,G_n)\}$ comprise those pairs in $\mathfrak{F}$ with $G_i(\vt_i)\notin K(\vt_i)^{\times 2}$. Then $\rho(\mathfrak F')=0$ and $c(\mathfrak F') = c(\mathfrak F)\leq 3$. Clearly, the strongly $\mathfrak F$-admissible triplet $\c P$ is also strongly $\mathfrak F'$-admissible.

\begin{lemma}
\label{lem:divisor reduction}
Let $\rho(\mathfrak{F}) <n$. Then, for any $\epsilon \in (0,1)$, the sum $D(\mathfrak{F},f,\c{P};X)$ is $\gg$ 

\begin{equation}\label{eq:divisor reduction}
\sum_{\substack{\ddd_1,\ldots,\ddd_{\rho(\mathfrak{F})}\in\ideals\\\norm\ddd_i \leq X^\epsilon\ \forall i\\
\ddd_i+\www=\OO_K\ \forall i\\
\ddd_i+\ddd_j=\OO_K\ \forall i \neq j
}}
\left(\prod_{i=1}^{\rho(\mathfrak{F})}\mu_K^2(\ddd_i)\one_{f'}(\ddd_i)\right)
\hspace{-0.2cm}
\sum_{\substack{(\sigma_i,\tau_i)\bmod{\ddd_i}\ \forall i\\\sigma_i\OO_K + \tau_i\OO_K + \ddd_i = \OO_K\\ 
F_{i}(\sigma_i,\tau_i)\equiv 0 \bmod{\ddd_i}\ \forall i}}
\sum_{\substack{(s,t)\in M^*(\c{P},X)\\ 
(s,t) \equiv (\sigma_i,\tau_i)\bmod{\ddd_i}\  \forall i
}}
\hspace{-0,5cm}
r(\mathfrak{F}',f,\c{P}; s,t).
\end{equation}
In these sums, the quantifiers $\forall i$ run over all $i\in \{1,\ldots,\rho(\mathfrak F)\}$.
\end{lemma}

\begin{proof}
This stems upon re-ordering the sum with respect to the factors $\ddd_i\mid F_i(s,t)$ and splitting into congruence classes $\bmod \ddd_i$. Since $r(s,t)\geq 0$, we are allowed to impose additional restrictions on the $\ddd_i$, such as $\norm\ddd_i \leq X^\epsilon$.  
\end{proof}

\begin{lemma}\label{lem:aux 1}
  Let $\classrep,\aaa \in\ideals$, $\classrep\mid\aaa$, and let $(\tilde{\sigma},\tilde{\tau})\in\classrep^2$ such that $\tilde{\sigma}\OO_K + \tilde{\tau}\OO_K + \aaa = \classrep$. Then there is $(\sigma,\tau)\in\classrep^2$ satisfying $(\sigma,\tau)\equiv (\tilde{\sigma},\tilde{\tau})\bmod\aaa$ and $\sigma\OO_K + \tau\OO_K = \classrep$.
\end{lemma}

\begin{proof}
  Let $\bbb\in\ideals$ such that $\bbb\aaa=w\OO_K$ is a principal ideal, and such that any prime ideal $\ppp$ dividing $\tilde{\sigma}$ divides $\bbb$ if and only if it does not divide $\tilde{\tau}\classrep^{-1}$. We may then choose $\sigma:=\tilde{\sigma}$ and $\tau := \tilde{\tau} + w$.
\end{proof}
We next deploy~Theorem \ref{thm:main 2}
to
estimate the innermost sum in Lemma~\ref{lem:divisor reduction}.
\begin{lemma}\label{lem:apply crt}
Let $\rho(\mathfrak{F})<n$. There is a function $f_0\in \c{Z}_K$ and $\beta_0,\beta_1,\beta_2>0$, such that the following holds: for any $\ddd_1,\ldots,\ddd_{\rho(\mathfrak F)}\in\ideals$ and $(\sigma_i,\tau_i)\bmod \ddd_i$, satisfying the conditions under the first two sums in \eqref{eq:divisor reduction}, we have, with $\ddd:=\ddd_1\cdots\ddd_{\rho(\mathfrak F)}$, the asymptotic
 \begin{equation}\label{eq:prop result}
\sum_{\substack{(s,t)\in M^*(\c{P},X)\\ 
(s,t) \equiv (\sigma_i,\tau_i)\bmod{\ddd_i}\  \forall i
}}
\hspace{-0,5cm}
r(\mathfrak{F}',f,\c{P}; s,t) = 
  \beta_0X^2\frac{\one_{f_0}(\ddd)}{\norm\ddd^2} + O(X^{2-\beta_1}\norm\ddd^{\beta_2}).
 \end{equation}
The implicit constant in the error term is independent
of all
$\ddd_i$, $(\sigma_i,\tau_i)$.
\end{lemma}

\begin{proof}
  The Chinese remainder theorem and the coprimality conditions on $\ddd_1,\ldots,\ddd_{\rho(\mathfrak F)}, \www$ allow us to express the congruences $(s,t)\equiv (\sigma,\tau)\bmod\www$ and $(s,t)\equiv (\sigma_i,\tau_i)\bmod\ddd_i$ for all $i$ as one congruence $(s,t)\equiv (\tilde{\sigma},\tilde{\tau})\bmod \ddd\www$. The pair $(\tilde{\sigma},\tilde{\tau})\in\OO_K^2$ then necessarily satisfies $\tilde{\sigma}\OO_K + \tilde{\tau}\OO_K + \ddd\www = \classrep$. Using Lemma \ref{lem:aux 1}, we may thus assume that $\tilde{\sigma}\OO_K + \tilde{\tau}\OO_K = \classrep$. 

The triplet $\c P' := (\c D, (\tilde\sigma,\tilde\tau), \www)$ is strongly $\mathfrak F'$-admissible. Moreover $\ddd$ satisfies the condition \eqref{eq:asympt thm divisor cond} in Theorem \ref{thm:main 2}, since $\classrep\prod_{i, j}\res(F_i,F_j)|\www$, and since $\ddd_i+\www=\OO_K$ for all $i$.

The sum in the lemma equals
\begin{equation*}
  \sum_{(s,t)\in M^*(\c P_\ddd', X)}r(\mathfrak F',f,\c P'; s,t), 
\end{equation*}
so the lemma 
stems
from Theorem \ref{thm:main 2}, once we enlarge $S_\bad$ and replace $\www$ by a sufficiently high power to ensure that $\www_0\mid \www$. 
\end{proof}

Using the bound $|\one_{f'}(\ddd_i)|\ll \norm\ddd_i$, we see that the error terms arising from substituting \eqref{eq:prop result} into \eqref{eq:divisor reduction} are $\ll X^{2-\beta_1+\epsilon\rho(\mathfrak F)(\beta_2+3)}$. 
Finally,
choosing $\epsilon$ small enough makes
the exponent smaller than $2$. 

Let us consider the main term. For a form $F \in \OO_K[s,t]$, irreducible over $K$ and not divisible by $t$
and for
$\ddd\in\ideals$ we define
\begin{equation}\label{eq:tau F defin}
\uptau_F(\ddd):=\card\{\mu\in\OO_K/\ddd 
:
F(\mu,1)\equiv 0\bmod \ddd \}.
\end{equation}
Using \eqref{eq:w leading coeffs}, we obtain
for all $\ddd\in\ideals$ with $\ddd+\www=\OO_K$,
\begin{equation*}
  \sum_{\substack{
(\sigma,\uptau)\bmod{\ddd}\\
F(\sigma,\tau)\equiv 0 \bmod{\ddd}
\\
\sigma\OO_K + \tau\OO_K + \ddd = \OO_K
}} \hspace{-0,4cm}
1
= \uptau_{F}(\ddd)\phi_K(\ddd).
\end{equation*}
Let us now introduce the function
\begin{equation*}
L(\ddd):=\one_{f'}(\ddd)\one_{f_0}(\ddd)\norm(\ddd)^{-1}\phi_K(\ddd)\sum_{\ddd_1\cdots \ddd_{\rho(\mathfrak F)}=\ddd}\prod_{i=1}^{\rho(\mathfrak{F})}\uptau_{F_{i}}(\ddd_i).
\end{equation*}
To finish the proof of Theorem \ref{thm:main 1} in the case $\rho(\mathfrak{F})<n$, it remains to show that
\begin{equation*}
\sum_{\substack{\norm\ddd \leq X^\epsilon \\ \ddd+\www=\OO_K}}\mu_K^2(\ddd)\frac{L(\ddd)}{\norm\ddd}\gg(\log X)^{\rho(\mathfrak{F})}.
\end{equation*}
This bound
can be proved in a straightforward manner
by alluding to
the generalisation 
of Wirsing's theorem
to all number fields as
supplied in~\cite[Lemma 2.2]{ffrei}.
The required estimate
\[\sum_{\norm\ppp\leq X}
\frac{\uptau_{F_{i}}(\ppp)}{\norm\ppp}
\log \norm\ppp=\log X +O(1)
\]
follows from the prime ideal theorem for the number field $K(\vt_i)$.

Finally, if $\rho(\mathfrak{F})=n$, we proceed as in Lemma \ref{lem:divisor reduction} to obtain a lower bound for $D(\mathfrak{F},f,\c P;s,t)$ as in \eqref{eq:divisor reduction}, but with $r(\mathfrak{F}',f,\c P; s,t)$ replaced by $1$. Arguing
as in Lemma \ref{lem:apply crt}
and using M\"obius inversion as
in the proof of 
Lemma \ref{lem:counting with coprime}, 
the innermost sum then becomes
\begin{equation*}
  \sum_{(s,t)\in M^*(\c P'_\ddd,X)}1 = \sum_{\substack{\aaa\in\ideals\\\aaa+\ddd\www\classrep^{-1}=\OO_K\\\norm \aaa\ll X}}\card\left(((\sigma^*,\tau^*)+(\aaa\classrep\ddd\www)^2) \cap X^{1/m}\mathcal{D}\right),
\end{equation*}
for some $(\sigma^*,\tau^*)\in\OO_K^2$. By lattice point counting, the summand for $\aaa$ is
\begin{equation*}
\card\left((\aaa\classrep\ddd\www)^2 \cap (-(\sigma^*,\tau^*)+X^{1/m}\mathcal{D})\right)=\frac{c_K X^2\vol\c D}{\norm(\aaa\classrep\ddd\www)^2}+O\left(\left(\frac{X}{\norm\aaa}\right)^{2-1/m} + 1\right).
\end{equation*}
Summing this over all $\aaa$ yields 
a positive constant
$\beta_0=\beta_0(\classrep,\c D,\www)$,
such that  
\begin{equation*}
\sum_{(s,t)\in M^*(\c P'_\ddd,X)}1 = \beta_0
\frac{X^2}{\norm\ddd^2} + O(X^{2-1/m}\log X)
.\end{equation*}
We may use this asymptotic instead of Lemma~\ref{lem:apply crt} to proceed as in the case $\rho(\mathfrak{F})<n$. 
This completes our proof of Theorem \ref{thm:main 1}.

\section{Proof of Theorem \ref{thm:main 2}: 
Asymptotics for divisor sums}
\label{sec:divisor sum asymptotics}
Recall that we have 
shown that it is sufficient to consider the case 
when
none of the forms $F_i$ is proportional to $t$. The ideal $\www_0$ will be modified throughout the proof, but it will only depend on $K,\classrep,\mathfrak{F},f$. We start by assuming that $\www_0$ satisfies \eqref{eq:w leading coeffs}. Let $\mathfrak{F}$ be a system of forms as in the theorem, and $\www$ be a strongly $\mathfrak{F}$-admissible triplet with $\www_0\mid\www$. Moreover, let $\ddd\in\ideals$
satisfy \eqref{eq:asympt thm divisor cond}.

\subsection{The Dirichlet hyperbola trick}
\label{finalfinal}
Let us recall that the expression
\begin{equation*}
\sum_{\substack{(s,t)\in M^*(\c{P}_\ddd,X)}}
\hspace{-0,3cm}
r(\mathfrak{F},f,\c{P};s,t)
\end{equation*}
can be recast as
\begin{equation}
\label{eq:complete93}
\hspace{-0.5cm}
\sum_{\substack{(s,t)\in M^*(\c P_\ddd,X)}}
\hspace{0,1cm}
\prod_{i=1}^n\one_{f}(F_i(s,t)^\flat)\left(\sum_{\substack{\ccc_i\mid F_i(s,t)^\flat}}\qr{G_i(s,t)}{\ccc_i}\right).
\end{equation}
Defining $\www_i:=\prod_{\ppp\mid\www}\ppp^{v_\ppp(F_i(\sigma,\tau))}$
makes apparent,
once \eqref{eq:same power} has been taken into account,
that $F_i(s,t)^\flat=F_i(s,t)\www_i^{-1}$. 
Furthermore, for each  $(s,t)\in M^*(\c P_\ddd,X)$
we have the following
inequalities,
\begin{equation*}
  \norm F_i(s,t)^\flat =\norm\www_i^{-1}\prod_{v\in\archplaces}\absv{F_i(s,t)}^{\locdegv} \ll \prod_{v\in\archplaces}\max\{\absv{s},\absv{t}\}^{\locdegv\deg F_i} \ll X^{\deg F_i},
\end{equation*}
thus for each index $i$
there exists $c_i>0$, independent of $X$,
such that whenever
$X>1$
and 
$(s,t)\in M^*(\c P_\ddd,X)$ 
then
$\norm F_i(s,t)^\flat<c_iX^{\deg F_i}$.
We let 
$
Y_i := c_iX^{\deg F_i}
$.
Suppressing the dependence on $\www$ in the notation, we define the arithmetic
functions
\begin{equation*}
r_{i}^{-}(s,t):=\sum_{\substack{\ccc_i\mid F_i(s,t)^\flat\\\norm\ccc_i< \sqrt{Y_i}}}\qr{G_i(s,t)}{\ccc_i}
\ 
\
\
\text{ and }
\
\ 
\
r_{i}^{+}(s,t):=
\hspace{-0,7cm}
\sum_{\substack{\ccc^\ast_i\mid F_i(s,t)^\flat\\\norm\ccc_i^\ast< \sqrt{Y_i}\\\norm c^{\ast}_i\sqrt{Y_i}\norm\www_i\leq\norm(F_i(s,t))}}
\hspace{-0,7cm}
\qr{G_i(s,t)}{\ccc^{\ast}_i},
\end{equation*}
an action which,
upon writing $F_i(s,t)^{\flat}=\ccc_i\ccc^\ast_i$
and using assumption~\eqref{eq:qr is one},
allows us to
obtain the validity of 
\begin{equation*}
\sum_{\substack{\ccc_i\mid F_i(s,t)^\flat}}\qr{G_i(s,t)}{\ccc_i}=r_{i}^{-}(s,t)+r_{i}^{+}(s,t)
.\end{equation*} 
Let us introduce
for every $\b{v} \in [0,\infty)^n$ and  $\boldsymbol{\psi}=\l(\psi_1,\ldots,\psi_n\r) \in \{0,1\}^n$ the region
\begin{equation}\label{eq:def dpsi}
\mathscr{D}_{\boldsymbol{\psi}}(X; \vv) := \bigcap_{i=1}^n \left\{
(s,t) \in X^{1/\dg}\mathcal{D}\where \norm(F_i(s,t))\geq \psi_iv_i\sqrt{Y_i}\norm\www_i\right\} \subseteq K_\infty^2.
\end{equation}
Here
$X$ is considered as fixed and the dependence on $\vv$ is what we are interested in. Define $\omega_{\vpsi}(X;\vv):\R^n\to\R$ 
through
\begin{equation}
\label{volume def}
\vv\mapsto \vol(\mathscr{D}_{\vpsi}(X;\vv)).
\end{equation} 
For $\vic=(\ccc_1,\ldots,\ccc_n)\in\ideals^n$ we use the abbreviation
 $\norm\vic := (\norm\ccc_1, \ldots, \norm\ccc_n)\in (0,\infty)^n$
and arrive at  the equality of the quantity in~\eqref{eq:complete93}
with
\[
\sum_{\substack{(s,t)\in M^*(\c P_\ddd,X)}} \prod_{i=1}^n\one_{f}(F_i(s,t)^\flat)(r_i(s,t)^{-}+r_i(s,t)^{+})
,\]
which can be reshaped into
\[ 
\sum_{\vpsi \in \{0,1\}^n}
\sum_{\substack{\vic\in \ideals^n\\\norm\ccc_i < \sqrt{Y_i}\ \forall i \\ \prod_{i=1}^n \ccc_i + \ddd\www=\OO_K 
\\ \ccc_i + \ccc_j=\OO_K\ \forall i\neq j}}  
\sum_{\substack{(s,t)\in M^*(\c P_\ddd, X) \\ (s,t)\in \mathscr{D}_{\vpsi}(X; \norm\vic)
\\
\ccc_i|F_i(s,t)\ \forall i}}
\
\prod_{i=1}^n\qr{G_i(s,t)}{\ccc_i}\one_{f}(F_i(s,t)^{\flat})
.\]
Here we added the coprimality condition $\prod_{i=1}^n \ccc_i + \ddd\www=\OO_K$ due to \eqref{eq:asympt thm divisor cond} and the assumptions $\ccc_i + \ccc_j =\OO_K$ for $i\neq j$ due to \eqref{eq:w leading coeffs}. The identity
\begin{equation*}
\one_{f}(F_i(s,t)^{\flat})=\sum_{\substack{\bbb_i|F_i(s,t)\\\bbb_i+\www=\OO_K}}f(\bbb_i)
\end{equation*}
reveals that, with
\begin{equation*}
S_{\vpsi}:=
\hspace{-0.4cm}
\sum_{\substack{\vib, \vic  \in  \ideals^n\\
\norm\bbb_i< Y_i, \norm\ccc_i < \sqrt{Y_i}\ \forall i\\\prod_{i=1}^n \bbb_i\ccc_i+\ddd\www=\OO_K 
 \\\ccc_i+\ccc_j = \bbb_i + \bbb_j= \bbb_i+\ccc_j = \OO_K \ \forall i\neq j 
}}
\hspace{-0,5cm}
\prod_{i=1}^n f(\bbb_i)\sum_{\substack{
(s,t)\in \mathscr{D}_{\vpsi}(X; \norm\vic)
\\
(s,t)\in M^*(\c P_\ddd, X) 
\\(\bbb_i\cap \ccc_i)\mid F_i(s,t)\ \forall i}}
\
\
\prod_{i=1}^n\qr{G_i(s,t)}{\ccc_i},
\end{equation*}
one has
\begin{equation}\label{eq:opleiding}
\sum_{\substack{(s,t)\in M^*(\c P_\ddd,X)}} r(\mathfrak{F},f,\c P;s,t)=\sum_{\boldsymbol{\psi} \in \{0,1\}^n}S_{\vpsi}.
\end{equation}
For any $\aaa \in \ideals$
we let $\langle \aaa \rangle \subset \ideals$ denote the monoid generated by the prime ideals dividing $\aaa$. We collect here some conditions on $n$-tuples $\via, \vib''', \vic'', \vic'''\in\ideals^n$ for later reference:
\begin{align}\label{eq:radical1d}
\forall i:\ &\aaa_i+\ddd\www=\OO_K \ \text{ and }\ \aaa_i + \prod_{j<i}\aaa_j = \OO_K,\\
\label{eq:radical1}
\forall i:\ &\norm\aaa_i\bbb'''_i < Y_i,\quad \bbb'''_i+\ddd\www \prod_{j=1}^n \aaa_j \ccc'''_j=\OO_K
\ \text{ and }\ \bbb'''_i + \prod_{j<i}\bbb'''_j=\OO_K,\\
\label{eq:radical1c}
\forall i:\ &\norm\aaa_i\ccc''_i\ccc'''_i < \sqrt{Y_i},\quad \ccc''_i\in\langle \aaa_i \rangle,\quad \ccc'''_i+\ddd\www\prod_{j=1}^n \aaa_j=\OO_K\ \text{ and }\ \ccc'''_i + \prod_{j<i}\ccc'''_j=\OO_K.
\end{align}
Recall the
definition of  $\Lambda^*(\aaa,(\sigma,\tau),\ddd,\gamma)$ in \eqref{eq:def lattice}.
\begin{lemma}\label{lem:rec}
Write $\ddd_i:= \aaa_i \bbb'''_i\ccc''_i\ccc'''_i$,
$\ddd':=\prod_{i=1}^n \ddd_i$
and
let 
$\lambda$ be the, unique modulo $\ddd'$, 
solution of the system $\lambda \equiv \lambda_i \bmod \ddd_i$ for all $i$. 
Then the sum 
$S_{\vpsi}$ equals
\begin{align*}
\sum_{\substack{\via,\vib''',\vic'',\vic'''\in \ideals^n \\\eqref{eq:radical1d},\eqref{eq:radical1},\eqref{eq:radical1c}}}
\left(\prod_{i=1}^n f(\aaa_i \bbb'''_i)\right)
&\sum_{\substack{\lambda_i \bmod{\ddd_i}\ \forall i
\\ 
\ddd_i \mid F_i(\lambda_i,1)}}
\hspace{0.2cm}
\left(\prod_{i=1}^n\qr{G_i(\lambda_i,1)}{\aaa_i\ccc''_i\ccc'''_i}\right)
\cdot
\\ 
\cdot
&\left|\Lambda^*(\ddd\www,(\sigma,\tau),\ddd',\lambda) \cap\mathscr{D}_{\vpsi}(X;(\norm\aaa_i \ccc''_i \ccc'''_i)_{i=1}^n) \right|.
\end{align*}
\end{lemma}

\begin{proof}
For each
pair of ideals $\bbb_i,\ccc_i$ in the definition of $S_{\vpsi}$
we let $\aaa_i:= \bbb_i+\ccc_i$.
Therefore
$\bbb_i=\aaa_i \bbb'_i$ and $\ccc_i=\aaa_i \ccc'_i$ for some coprime ideals $\bbb'_i,\ccc'_i$ 
which satisfy $\bbb_i\cap \ccc_i=\aaa_i\bbb'_i\ccc'_i$.
We may further decompose $\bbb'_i$ and $\ccc'_i$ uniquely as
$\bbb'_i=\bbb''_i\bbb'''_i$ and $\ccc'_i=\ccc''_i\ccc'''_i$,
where 
$\bbb''_i,\bbb'''_i,\ccc''_i,\ccc'''_i \in \ideals$
and for all non-zero prime ideals $\ppp$ we have
\begin{equation*}
\ppp|\bbb''_i\ccc''_i \Rightarrow \ppp|\aaa_i\ \ \text{ and }\ \ \ppp|\bbb'''_i\ccc'''_i \Rightarrow \ppp\nmid \aaa_i.
\end{equation*}
Since the function $f$ is supported on square-free ideals, the only relevant value for $\bbb''_i$ in 
$S_{\vpsi}$ is $\bbb''_i=\OO_K$. 
Taking into account
the conditions \eqref{eq:radical1d}, \eqref{eq:radical1} and \eqref{eq:radical1c}
we have thus  
obtained the following factorization for the $\bbb_i,\ccc_i$ in the sum $S_{\vpsi}$,
\begin{equation*}
\bbb_i= \aaa_i\bbb'''_i \ \text{ and }\ \ccc_i=\aaa_i \ccc''_i \ccc'''_i
.\end{equation*}
We are therefore led to the equality
of
$S_{\vpsi}$ with
\begin{equation*}
\sum_{\substack{\via, \vib''',\vic'',\vic''' \in  \ideals^n \\\eqref{eq:radical1d},\eqref{eq:radical1},\eqref{eq:radical1c}}} \left(\prod_{i=1}^n f(\aaa_i \bbb'''_i)\right)
\sum_{\substack{
(s,t)\in  M^*(\c P_\ddd,X)
\\
(s,t) \in \mathscr{D}_{\vpsi}(X;(\norm \aaa_i \ccc''_i \ccc'''_i)_{i=1}^n)
\\ \aaa_i\bbb'''_i\ccc''_i\ccc'''_i|F_i(s,t)\ \forall i}}\prod_{i=1}^n \qr{G_i(s,t)}{\aaa_i\ccc''_i\ccc'''_i}.
\end{equation*}
For any pair $(s,t)$ in the inner sum we have
 $t\OO_K + \ddd_i = \OO_K$,
since if $\ppp\mid t\OO_K + \ddd_i$ then $\ppp\nmid\www$
and hence $\ppp\nmid F_i(1,0)$. This implies that $\ppp\mid s$ and thus $\ppp \mid s\OO_K + t\OO_K = \classrep \mid \www$, a contradiction. Hence, letting $\lambda_i:=st^{-1}\bmod{\ddd_i}$ we obtain the congruence 
$s\equiv \lambda_i t \bmod{\ddd_i}$. 
Note that
each
$G_i$ has even degree
and therefore  
\begin{equation*}
\left(
\frac{G_i(s,t)}{\aaa_i\ccc''_i\ccc'''_i}
\right)=\left(
\frac{G_i(\lambda_i,1)}{\aaa_i\ccc''_i\ccc'''_i}
\right),
\end{equation*}
an equality which can be exploited to 
transform
$S_{\vpsi}$
into
\begin{equation*}
\sum_{\substack{\via,\vib''',\vic'',\vic''' \in  \ideals^n \\\eqref{eq:radical1d},\eqref{eq:radical1},\eqref{eq:radical1c}}} \left(\prod_{i=1}^n f(\aaa_i \bbb'''_i)\right)\sum_{\substack{\lambda_i \bmod{\ddd_i}\ \forall i\\\ddd_i \mid F_i(\lambda_i,1)}}\vspace{0.2cm}\prod_{i=1}^n\qr{G_i(\lambda_i,1)}{\aaa_i\ccc''_i\ccc'''_i}
\sum_{\substack{
(s,t)\in M^*(\c P_\ddd, X)
\\ 
(s,t)\in  \mathscr{D}_{\vpsi}(X;(\norm \aaa_i \ccc''_i \ccc'''_i)_{i=1}^n)
\\s\equiv \lambda_i t \bmod{\ddd_i}\ \forall i}}1
.\end{equation*} 
Since the $\ddd_i$ are relatively prime in pairs, we may combine the congruences under the innermost sum to a single congruence of the form $s\equiv \lambda t \bmod{\ddd'}$
and our lemma is furnished
upon tautologically reformulating 
the innermost sum.
\end{proof}

\subsection{Application of lattice point counting}
\label{s:minimal application}
Let us 
define the multiplicative function 
on $\ideals$,
\begin{equation*}
\upeta(\aaa):=\frac{\mu_K(\aaa)}{\norm\aaa}\prod_{\ppp|\aaa}\left(1+\frac{1}{\norm\ppp}\right)^{-1},
\end{equation*}
which 
is supported on square-free
ideals and satisfies $|\upeta(\ppp)| < 1/\norm\ppp$ for all prime ideals $\ppp$. We use the symbols $\ddd_i,\ddd',\lambda$ with the same meaning as in Lemma \ref{lem:rec}. For any $\vpsi \in \{0,1\}^n$, let

\begin{equation*}
M_{\psi}:=\sum_{\substack{\via,\vib''',\vic'',\vic''' \in  \ideals^n \\\eqref{eq:radical1d},\eqref{eq:radical1},\eqref{eq:radical1c}}}\hspace{-0,5cm}\omega_{\vpsi}(X;(\norm\aaa_i \ccc''_i \ccc'''_i)_{i=1}^n) \prod_{i=1}^n \left(\frac{f(\aaa_i \bbb'''_i) \one_{\upeta}(\aaa_i\bbb'''_i\ccc'''_i)}{\norm\aaa_i\bbb'''_i\ccc''_i\ccc'''_i}\sum_{\substack{\lambda_i \bmod{\ddd_i}\\\ddd_i |F_i(\lambda_i,1)}}\qr{G_i(\lambda_i,1)}{\aaa_i\ccc''_i\ccc'''_i}\right).
\end{equation*}

\begin{lemma}
\label{cramped}
Let $Y:=\prod_{i=1}^n Y_i$. Then, for all $\epsilon>0$, we have
\begin{equation*}
\sum_{\substack{(s,t)\in  M^*(\c P_\ddd, X)}} 
\hspace{-0,3cm}
r(\mathfrak{F},f,\c P; s,t)=
\frac{c_K}{\norm(\ddd\www)^2}\prod_{\ppp|\ddd\www\classrep^{-1}}
\hspace{-0,2cm}
\left(1-\frac{1}{\norm\ppp^2}\right)^{-1}
\hspace{-0,3cm}
\sum_{\vpsi \in \{0,1\}^n}M_{\vpsi} + O_{\epsilon}(X^{2-1/(4m)+\epsilon}).
\end{equation*}
Here, $c_K$ is a positive constant depending only on $K$
and the implied constant in the error term depends only on $K,\classrep,\c D,\www, \mathfrak{F},f,\epsilon$.
\end{lemma}

\begin{proof}
Recall that $\mathcal{C} = \{\classrep_1, \ldots, \classrep_h\}$ is a fixed system of integral representatives of the class group of $K$. By possibly modifying $\www_0$, we may assume that $\classrep_1\cdots \classrep_h \mid \www$.

Since $\c{D}\subseteq K_\infty^2=\R^{2\dg}$ is a cartesian product of balls in $K_v^2 = \R^{2\locdegv}$, it is clear that the sets $\c D_\vpsi(X;\b v)\subseteq \R^{2\dg}$, for $X>0$ and $\b v\in \R^n$ are fibres of a definable family with parameters $(X,\b v,\vpsi)\in\R^{1+2n}$ in the o-minimal structure $\R_{\text{alg}}$ of semialgebraic sets. Moreover, $\c D_\vpsi(X;\b v)\subseteq X^{1/m}\c D$, which is contained in a zero-centered ball of radius $\ll X^{1/m}$.

Injecting the estimate of Lemma \ref{lem:counting with coprime} into Lemma \ref{lem:rec} yields the desired main term.
The sum over the error terms in Lemma \ref{lem:counting with coprime} can be bounded by $\ll E_0 + \cdots + E_{m-1}$, where, for $0\leq j\leq m-1$,
\begin{align*}
  E_j := \sum_{\substack{\via,\vib''',\vic'',\vic''' \in  \ideals^n \\ \norm \aaa_i \bbb'''_i \leq Y_i \\ \norm \aaa_i \ccc''_i \ccc'''_i \leq \sqrt{Y_i}\\\aaa_i\bbb_i'''\ccc_i''\ccc_i'''+\www=\OO_K\\ \bbb_i'''+\aaa_i\ccc_i''\ccc_i'''=\OO_K}}\prod_{i=1}^n \frac{1}{\norm\aaa_i \bbb'''_i} \sum_{\substack{\lambda_i \bmod{\ddd_i}\ \forall i\\\ddd_i \mid F_i(\lambda_i,1)}}\frac{X^{1+j/m+\epsilon}}{\min_{1\leq q\leq h}\{\lambda^{(1)}(\classrep_q,\classrep_q\ddd',\lambda)^m\lambda^{(m+1)}(\classrep_q,\classrep_q\ddd',\lambda)^j\}}.
\end{align*}

Let us bound $E_j$. The Chinese remainder theorem allows us to separate the sum over $\lambda_i\bmod\ddd_i$ into a sum over $\lambda_i\bmod\aaa_i\ccc_i''\ccc_i'''$ and a sum over $\lambda_i\bmod \bbb_i'''$.
Write $\ddd'':=\prod_{i=1}^n\aaa_i\ccc_i''\ccc_i'''$
and let $\lambda'\equiv \lambda_i\bmod\aaa_i\ccc_i''\ccc_i'''$ for all $i$. Since $\Lambda(\classrep_q,\classrep_q\ddd',\lambda)\subset \Lambda(\classrep_q,\classrep_q\ddd'',\lambda')$, we obtain
\begin{equation*}
  \lambda^{(i)}(\classrep_q,\classrep_q\ddd',\lambda)\geq \lambda^{(i)}(\classrep_q,\classrep_q\ddd'',\lambda')
\end{equation*}
for all $1\leq i\leq 2m$. This allows us to sum over $\vib'''$, obtaining the estimate
\begin{equation}\label{eq:ej bound}
  E_j \ll\hspace{-0.6cm} \sum_{\substack{\via,\vic'',\vic''' \in  \ideals^n \\\norm \aaa_i \ccc''_i \ccc'''_i \leq \sqrt{Y_i}\\\aaa_i\ccc_i''\ccc_i'''+\www=\OO_K}}\prod_{i=1}^n \frac{1}{\norm\aaa_i} \sum_{\substack{\lambda_i \bmod{\aaa_i\ccc_i''\ccc_i'''}\ \forall i\\\aaa_i\ccc_i\ccc_i''' \mid F_i(\lambda_i,1)}}\frac{X^{1+j/m+\epsilon}}{\min_{1\leq q\leq h}\{\lambda^{(1)}(\classrep_q,\classrep_q\ddd'',\lambda')^m\lambda^{(m+1)}(\classrep_q,\classrep_q\ddd'',\lambda')^j\}}.
\end{equation}
Each first successive minimum $\lambda^{(1)}(\classrep_q,\classrep_q\ddd'',\lambda')$ is attained by a point $\vv = (v_1,v_2)$ in the lattice $\Lambda(\classrep_q,\classrep_q\ddd'', \lambda')\subseteq \OO_K^2\subset K_\infty^2$, of euclidean norm bounded by
\begin{equation*}
  \vecnorm{\vv}\ll \norm\ddd''^{1/(2m)} \ll Y^{1/(4m)}, 
\end{equation*}
due to Lemma \ref{lem:lattice}. Let
\begin{equation*}
  E_j(\vv):=\sum_{q=1}^{h}\sum_{\substack{\via,\vic'',\vic''' \in  \ideals^n \\ \norm \aaa_i \ccc''_i \ccc'''_i \leq \sqrt{Y_i}\\\aaa_i\ccc_i''\ccc_i'''+\www=\OO_K}}\prod_{i=1}^n \frac{1}{\norm\aaa_i} \sum_{\substack{\lambda_i \bmod{\aaa_i\ccc_i''\ccc_i'''}\ \forall i\\\aaa_i\ccc_i''\ccc_i''' \mid F_i(\lambda_i,1)\\ \vv \in \Lambda(\classrep_q,\classrep_q\ddd'',\lambda')\\\vecnorm{\vv}=\lambda^{(1)}(\classrep_q,\classrep_q\ddd'',\lambda')}}\frac{1}{\vecnorm{\vv}^m\lambda^{(m+1)}(\classrep_q,\classrep_q\ddd'',\lambda')^{j}}.
\end{equation*}
Sorting the expression in \eqref{eq:ej bound} by the first successive minimum, we see that
\begin{equation*}
  E_j \ll \sum_{\substack{\vv\in\OO_K^2\smallsetminus\{0\}\\\vecnorm{\vv}\ll Y^{1/(4m)}}}X^{1+j/m+\epsilon}E_j(\vv).
\end{equation*}
For $\vv\in\OO_K^2$ to be an element of the lattice $\Lambda(\classrep_q,\classrep_q\ddd'', \lambda')$, it is necessary that $v_1 \equiv \lambda' v_2 \bmod \ddd''$, so in particular $v_1 \equiv \lambda_i v_2 \bmod \aaa_i\ccc_i''\ccc_i'''$
and hence $\aaa_i\ccc_i''\ccc_i''' \mid F_i(\vv)$. This allows us to conclude that
\begin{equation*}
  E_j(\vv) \ll \frac{\prod_{i=1}^n\norm(F_i(\vv))^{\epsilon}}{\vecnorm{\vv}^{m+j}} \ll \frac{X^{\epsilon}}{\vecnorm{\vv}^{m+j}},
\end{equation*}
whenever $F_i(\vv)\neq 0$ holds for all $1\leq i\leq n$. The sum of $E_j(\vv)$ over all such $\vv$ is 
\begin{equation*}
  \ll X^{1+j/m+\epsilon}\sum_{\substack{\vv\in\OO_K^2\smallsetminus\{0\}\\\vecnorm{\vv}\ll Y^{1/(4m)}}}\frac{1}{\vecnorm{\vv}^{m+j}} \ll X^{1+j/m+\epsilon} Y^{1/2(1-(m+j)/(2m))} \ll  X^{1+j/m+\epsilon}Y^{1/4-j/(4m)}.
\end{equation*}
Recalling our assumption that $c(\mathfrak{F})\leq 3$ and the fact that $Y \ll X^{c(\mathfrak{F})}$, we see that this error term does not exceed
\begin{equation*}
  X^{2-1/4 + j/(4m)+\epsilon} \leq X^{2 - 1/(4m) + \epsilon}.
\end{equation*}
It remains to bound the sum over those $\vv$ for which $F_k(\vv)=0$ for some $1\leq k\leq n$. Since $F_k(s,t)$ is irreducible, this necessarily implies that $F_k(s,t)$ is linear
and since the forms $F_i(s,t)$ are pairwise coprime, we conclude that $F_i(\vv)\neq 0$ for all $i\neq k$. This allows us to bound the number of $\aaa_i,\ccc_i'',\ccc_i''',\lambda_i$, for $i\neq k$, as before by $\prod_{i\neq k}\norm(F_i(\vv))^{\epsilon}\ll X^\epsilon$. Writing temporarily
\begin{equation*}
  F_k(s,t) = a s - b t,
\end{equation*}
with $a\neq 0$ and $a\mid\www_0\mid\www$, we see that the
equality
$F_k(\lambda_k,1)\equiv 0\bmod \aaa_k\ccc_k''\ccc_k'''$ 
is equivalent to
$\lambda_k = a^{-1}b \bmod \aaa_k\ccc_k''\ccc_k'''$.
Moreover, $\Lambda(\classrep_q,\classrep_q\ddd'',\lambda') \subseteq \Lambda(\classrep_q,\classrep_q\aaa_k\ccc_k''\ccc_k''',\lambda_k)$. 
We may thus bound
\begin{equation*} E_j(\vv)\ll\sum_{q=1}^h\sum_{\substack{\aaa_k,\ccc_k'',\ccc_k'''\in\ideals\\\norm\aaa_k\ccc_k''\ccc_k'''\ll \sqrt{X}\\\aaa_k\ccc_k''\ccc_k'''+\www=\OO_K}} \frac{X^\epsilon}{\vecnorm{\vv}^m\lambda^{(m+1)}(\classrep_q,\classrep_q\aaa_k\ccc_k''\ccc_k''',\lambda_k)^j}.
\end{equation*}
Let $\alpha_1, \ldots, \alpha_m$ be $\Z$-linearly independent elements of $\classrep_q$ with $\vecnorm{\alpha_i}\asymp\lambda^{(i)}(\classrep_q) \asymp 1$ and let $\beta_1, \ldots, \beta_m$ be $\Z$-linearly independent in $\classrep_q\aaa_k\ccc_k''\ccc_k'''$ with $\vecnorm{\beta_i}\asymp\lambda^{(i)}(\classrep_q\aaa_k\ccc_k''\ccc_k''')\asymp \norm(\aaa_k\ccc_k''\ccc_k''')^{1/m}$. To estimate the successive minima, we used Minkowski's second theorem and the fact that  $\lambda^{(1)}(\aaa)\gg\norm\aaa^{1/m}$ holds for any $\aaa\in\ideals$ (see, e.g. \cite[Lemma 5]{MR2247898} or \cite[Lemma 5.1]{ANT}). This provides us with the linearly independent lattice points
\begin{equation*}
  \begin{pmatrix}
    b\alpha_1\\
    a\alpha_1
  \end{pmatrix}, \ldots,
  \begin{pmatrix}
    b\alpha_m\\
    a\alpha_m
  \end{pmatrix}, 
  \begin{pmatrix}
    \beta_1\\
    1
  \end{pmatrix}, \ldots,
  \begin{pmatrix}
    \beta_m\\
    1
  \end{pmatrix} \in \Lambda(\classrep_q,\classrep_q\aaa_k\ccc_k''\ccc_k''',\lambda_k).
\end{equation*}
The first $m$ of these have norm $\asymp 1$, whereas the latter $m$ ones have norm $\asymp \norm(\aaa_k\ccc_k''\ccc_k''')^{1/m}$, so the product of their norms is $\asymp\norm(\aaa_k\ccc_k''\ccc_k''')\asymp\det\Lambda(\classrep_q,\classrep_q\aaa_k\ccc_k''\ccc_k''',\lambda_k)$. Using again Minkowski's second theorem, this shows that the successive minima of $\Lambda(\classrep_q,\classrep_q\aaa_k\ccc_k''\ccc_k''',\lambda_k)$ satisfy 
\begin{align*}
\lambda^{(1)}(\classrep_q,\classrep_q\aaa_k\ccc_k''\ccc_k''',\lambda_k), \ldots, \lambda^{(m)}(\classrep_q,\classrep_q\aaa_k\ccc_k''\ccc_k''',\lambda_k)&\asymp 1,\\ \lambda^{(m+1)}(\classrep_q,\classrep_q\aaa_k\ccc_k''\ccc_k''',\lambda_k), \ldots, \lambda^{(2m)}(\classrep_q,\classrep_q\aaa_k\ccc_k''\ccc_k''',\lambda_k)&\asymp \norm(\aaa_k\ccc_k''\ccc_k''')^{1/m}.
\end{align*}
As a result, we obtain the bound
\begin{equation*} E_j(\vv)\ll\sum_{\substack{\aaa_k,\ccc_k'',\ccc_k'''\in\ideals\\\norm\aaa_k\ccc_k''\ccc_k'''\ll \sqrt{X}}} \frac{X^\epsilon}{\vecnorm{\vv}^m\norm(\aaa_k\ccc_k''\ccc_k''')^{j/m}}.
\end{equation*}
In addition, we observe that any $\vv=(v_1,v_2)\in\OO_K^2$ with $F_k(\vv)=0$ is uniquely determined by $v_2$.
Consequently, 
\begin{align*}
\sum_{\substack{\vv\in\OO_K^2\smallsetminus\{0\}\\\vecnorm{\vv}\ll Y^{1/(4m)}\\F_k(\vv)=0}}&X^{1+j/m+\epsilon}E_j(\vv)\ll X^{1+j/m+\epsilon}\sum_{\substack{v_2\in\OO_K\smallsetminus\{0\}\\\vecnorm{v_2}\ll Y^{1/(4m)}}}\frac{1}{\vecnorm{v_2}^m}\sum_{\substack{\aaa_k,\ccc_k'',\ccc_k'''\in\ideals\\\norm\aaa_k\ccc_k''\ccc_k'''\ll \sqrt{X}}} \frac{1}{\norm(\aaa_k\ccc_k''\ccc_k''')^{j/m}}\\
&\ll X^{1+j/m+\epsilon}(\log Y)^m X^{1/2(1-j/m)+\epsilon}\ll X^{3/2+j/(2m)+\epsilon}\ll X^{2-1/(2m)+\epsilon}.
\end{align*}
\end{proof}

\subsection{Controlling
the main term}
\label{s:simplifi}

Let $\rho_i(\aaa):=\rho_{(F_i,G_i)}(\aaa)$,
as defined prior to Lemma \ref{lem:artinakos}
and moreover 
recall~\eqref{eq:tau F defin}. 
\begin{lemma}
\label{lem:decompo}
The arithmetic factor in the definition of $M_{\vpsi}$ decomposes as follows:
\begin{equation*}
\sum_{\substack{\lambda_i \bmod{\ddd_i}\\\ddd_i |F_i(\lambda_i,1)}}\qr{G_i(\lambda_i,1)}{\aaa_i\ccc''_i\ccc'''_i}=\rho_i(\aaa_i\ccc''_i)\uptau_{F_i}(\bbb'''_i)\rho_i(\ccc'''_i).
\end{equation*}
\end{lemma}
\begin{proof}
Recall that we set $\ddd_i = \aaa_i\bbb_i\ccc_i''\ccc_i'''$, and that the ideals $\aaa_i \ccc''_i, \bbb'''_i, \ccc'''_i$ are coprime in pairs due to \eqref{eq:radical1d},\eqref{eq:radical1}
and~\eqref{eq:radical1c}. The Chinese remainder theorem, jointly with
multiplicativity properties of the Jacobi symbol, yields
\begin{equation*}
  \sum_{\substack{\lambda_i \bmod{\ddd_i}\\\ddd_i |F_i(\lambda_i,1)}}\qr{G_i(\lambda_i,1)}{\aaa_i\ccc''_i\ccc'''_i} = \sum_{\substack{\lambda_i' \bmod\aaa_i\ccc_i''\\\aaa_i\ccc_i''\mid F_i(\lambda_i',1)}}\qr{G_i(\lambda_i',1)}{\aaa_i\ccc_i''}\sum_{\substack{\lambda_i''\bmod \bbb_i'''\\\bbb_i'''\mid F_i(\lambda_i'',1)}}1\sum_{\substack{\lambda_i'''\bmod\ccc_i'''\\\ccc_i'''\mid F_i(\lambda_i''',1)}}\qr{G_i(\lambda_i''',1)}{\ccc_i'''}
.\end{equation*}
\end{proof}

Letting $\BBB:=\ddd\www\prod_{j=1}^n \aaa_j \ccc'''_j$, we define $M(\via, \vic'',\vic''')$ as
\begin{equation*}
\hspace{-0,7cm}\sum_{\substack{\bbb_1''' \in  \ideals\\\norm\bbb'''_1< Y_1/\norm\aaa_1\\\bbb'''_1+\BBB=\OO_K}}\hspace{-0,4cm}\frac{\one_{\upeta}(\bbb'''_1)f(\bbb'''_1)\uptau_{F_1}(\bbb'''_1)}{\norm\bbb'''_1}\hspace{-0,7cm}
\sum_{\substack{\bbb_2''' \in  \ideals\\\norm\bbb'''_2 < Y_2/\norm\aaa_2\\\bbb'''_2+\BBB\bbb'''_1=\OO_K}}\hspace{-0,5cm}\frac{\one_{\upeta}(\bbb'''_2)f(\bbb'''_2)\uptau_{F_2}(\bbb'''_2)}{\norm\bbb'''_2}\ldots\hspace{-1,0cm}
\sum_{\substack{\bbb_n''' \in  \ideals\\\norm\bbb'''_n < Y_n/\norm\aaa_n\\\bbb'''_n+\BBB\prod_{j<n}\bbb'''_j=\OO_K}}\hspace{-0,7cm}
\frac{\one_{\upeta}(\bbb'''_n)f(\bbb'''_n)\uptau_{F_n}(\bbb'''_n)}{\norm\bbb'''_n},
\end{equation*}
a definition that makes the succeeding
equality valid,
\begin{equation}
\label{eq:tara}
M_{\vpsi}=\sum_{\substack{\via,\vic'',\vic''' \in  \ideals^n \\\eqref{eq:radical1d},\eqref{eq:radical1c}}}\hspace{-0,3cm}\omega_{\vpsi}(X;(\norm\aaa_i \ccc''_i \ccc'''_i)_{i=1}^n))M(\via,\vic'',\vic''')\prod_{i=1}^n \frac{f(\aaa_i)\one_{\upeta}(\aaa_i \ccc'''_i)\rho_i(\aaa_i\ccc''_i)\rho_i(\ccc'''_i)}{\norm\aaa_i\ccc''_i\ccc'''_i}.
\end{equation}

Let us 
bring into play the multiplicative function $\upgamma$, supported on square-free ideals, by letting
$\upgamma(\ppp):=0$ for $\ppp\mid\www$ and
in the remaining case,
$\ppp\nmid\www$,
we define
\begin{equation*}
\upgamma(\ppp):=-1+\left(1+\frac{(1+\upeta(\ppp))f(\ppp)}{\norm\ppp}\sum_{i=1}^n\uptau_{F_i}(\ppp)\right)^{-1}
.\end{equation*}
Including enough small prime ideals in the factorization of $\www_0$, we can ensure that $\one_\upgamma\in\c U_K$.

\begin{lemma}\label{lem:yama}
Let $\upgamma_0:=\prod_{\ppp\nmid \www}(1+\upgamma(\ppp))^{-1}$ and suppose that $\norm\aaa_i\leq Y_i$ for all $1\leq i \leq n$. Then
\begin{equation*}
M(\via,\vic'',\vic''')=\upgamma_0\one_{\upgamma}(\ddd)\prod_{i=1}^n\one_{\upgamma}(\aaa_i)\one_{\upgamma}(\ccc'''_i)+O_\epsilon\left(X^\epsilon\max_{i=1,\ldots,n}\left\{\frac{\norm\aaa_i}{Y_i}\right\}\right).
\end{equation*}
The implied constant is independent of $\via,\vic'',\vic''',\ddd$, and $X$.
\end{lemma}

\begin{proof}
The bound bestowed upon $f$ by \eqref{eq:class 1} shows that each sum over $\bbb_i'''$ in $M(\via,\vic'',\vic''')$ forms an absolutely convergent series. We may complete the summation step-by-step for $i=n, n-1, \ldots, 1$. The bounds
\begin{equation*}
\one_{\upeta}(\bbb_i'''),\ |\norm\bbb_i'''f(\bbb_i''')|,\ \uptau_{F_i}(\bbb_i''') \ll_\epsilon \norm\bbb_i'''^\epsilon
\
\
\text{and}
\
\
\sum_{\norm\bbb_i>Y_i/\norm\aaa_i}\frac{\norm\bbb_i^\epsilon}{\norm\bbb_i^2}\ll X^\epsilon\frac{\norm\aaa_i}{Y_i}
\end{equation*}
reveal that the error introduced by this process is
$\ll_{\epsilon}X^\epsilon\max\left\{\norm\aaa_i/Y_i:i=1,\ldots,n\right\}$,
thus acquiring the main term
\begin{equation*}
\hspace{-0,7cm}\sum_{\substack{\bbb_1''' \in  \ideals\\\bbb'''_1+\BBB=\OO_K}}\hspace{-0,4cm}\frac{\one_{\upeta}(\bbb'''_1)f(\bbb'''_1)\uptau_{F_1}(\bbb'''_1)}{\norm\bbb'''_1}\ 
\cdots
\hspace{-0,7cm}
\sum_{\substack{\bbb_n''' \in  \ideals\\\bbb'''_n+\BBB\prod_{j<n}\bbb'''_j=\OO_K}}\hspace{-0,7cm}
\frac{\one_{\upeta}(\bbb'''_n)f(\bbb'''_n)\uptau_{F_n}(\bbb'''_n)}{\norm\bbb'''_n}.
\end{equation*}
Grouping all $n$-tuples $\vib'''$ according to the value of $\bbb:=\prod_{i=1}^n\bbb_i'''$
and letting
\begin{equation*}
g(\bbb):= \one_{\upeta}(\bbb)\sum_{\substack{\vib'''\in \ideals^n \\ \bbb=\bbb_1'''\ldots \bbb_n'''\\\bbb'''_i+\bbb'''_j=\OO_K \ \forall i\neq j}}\hspace{-0,3cm}\prod_{i=1}^nf(\bbb'''_i)\uptau_{F_i}(\bbb'''_i),
\end{equation*}
the main term becomes
\begin{equation*}
\sum_{\substack{\bbb \in  \ideals\\\bbb+\BBB=\OO_K}}\hspace{-0,5cm}\frac{g(\bbb)}{\norm\bbb} = \prod_{\ppp\nmid\BBB}\left(1+\frac{g(\ppp)}{\norm\ppp}\right) = \prod_{\ppp\nmid\BBB}(1+\upgamma(\ppp))^{-1}.
\end{equation*}
Here, we used the observation that $1+\upgamma(\ppp)=\left(1+\frac{g(\ppp)}{\norm\ppp}\right)^{-1}$ holds for all $\ppp\nmid\www$. 
\end{proof}
We may now
plant
Lemma \ref{lem:yama} into \eqref{eq:tara}
to
show that 
$M_{\vpsi}$ equals
\[\upgamma_0\one_{\upgamma}(\ddd)\sum_{\substack{\via,\vic'',\vic''' \in \ideals^n \\\eqref{eq:radical1d},\eqref{eq:radical1c}}}\hspace{-0,3cm}\omega_{\vpsi}(X;(\norm\aaa_i \ccc''_i \ccc'''_i)_{i=1}^n))\prod_{i=1}^n \frac{f(\aaa_i)\one_{\upeta}(\aaa_i \ccc'''_i)\rho_i(\aaa_i\ccc''_i)\rho_i(\ccc'''_i)\one_{\upgamma}(\aaa_i)\one_{\upgamma}(\ccc'''_i)}{\norm\aaa_i \ccc''_i\ccc'''_i}
\]
up to an error of size
\[
\ll_\epsilon
X^\epsilon\sum_{\substack{\via,\vic'',\vic'''\in \ideals^n \\\eqref{eq:radical1d},\eqref{eq:radical1c}}} \hspace{-0,3cm}\omega_{\vpsi}(X;(\norm\aaa_i \ccc''_i \ccc'''_i)_{i=1}^n))\left(\prod_{i=1}^n \frac{|f(\aaa_i)|\one_{\upeta}(\aaa_i \ccc'''_i)\rho_i(\aaa_i\ccc''_i)\rho_i(\ccc'''_i)}{\norm \aaa_i \ccc''_i \ccc'''_i}\right)\max_{1\leq i\leq n}\left\{\frac{\norm\aaa_i}{Y_i}\right\}.
\]

Using the inequalities 
$Y_i\gg X$, 
$\max\{\one_\upeta(\aaa), \rho_i(\aaa), f(\aaa) \norm\aaa \}\ll_\epsilon \norm\aaa^{\epsilon}$, 
\[\max_{1 \leq i \leq n}\{\norm\aaa_i\} \leq \prod_{i=1}^n \norm \aaa_i,
\
\
\text{and}
\
\
\omega_{\vpsi}(X;(\norm\aaa_i\ccc''_i\ccc'''_i)_{i=1}^n))\leq \vol(X^{1/m}\mathscr{D})\ll X^2,\]
we find that the sum in the error term is 
\begin{equation*}
\ll_\epsilon X^{1+\epsilon}\sum_{\substack{\via,\vic'',\vic''' \in \ideals^n \\\norm\aaa_i \ccc''_i \ccc'''_i \leq \sqrt{Y_i}}}\prod_{i=1}^n \frac{1}{\norm\aaa_i \ccc''_i\ccc'''_i}\ll_\epsilon X^{1+\epsilon}.
\end{equation*}
To analyze the main term further, we define on $\ideals$ the multiplicative functions
\begin{equation*}
g_i(\ccc_i):=\sum_{\substack{\aaa_i,\ccc''_i,\ccc'''_i \in \ideals\\\aaa_i\ccc''_i\ccc'''_i=\ccc_i\\\ccc''_i \in \langle \aaa_i \rangle\\\aaa_i+\ccc'''_i=\OO_K}}f(\aaa_i)\one_{\upeta}(\aaa_i \ccc'''_i)\rho_i(\aaa_i\ccc''_i)\rho_i(\ccc'''_i)\one_{\upgamma}(\aaa_i\ccc'''_i),
\end{equation*}
which satisfy, for prime ideals $\ppp$ and positive integers $k$,
\begin{equation*}
g_i(\ppp^k)=\sum_{\substack{\alpha,\beta,\gamma \in \Z_{\geq 0}\\
\alpha+\beta+\gamma=k\\\beta>0\Rightarrow \alpha >0\\\alpha \gamma=0}}f(\ppp^\alpha)\one_{\upeta}(\ppp^{\alpha+\gamma})\rho_i(\ppp^{\alpha+\beta})\rho_i(\ppp^\gamma)\one_{\upgamma}(\ppp^{\alpha+\gamma}).
\end{equation*}
Since $f$ is
supported on square-free ideals
the only candidate values for 
$(\alpha,\beta,\gamma)$
are $(0,0,k)$ and $(1,k-1,0)$. Let us mention that 
the group structure of $\c U_K$ provides us with a function $\updelta$
fulfilling
$\one_f\cdot\one_\upeta\cdot\one_\upgamma = \one_\updelta$. 
We are therefore afforded with the equality
$
g_i(\ppp^k)=\rho_i(\ppp^k)\one_{\updelta}(\ppp^k)
$, which, upon introducing
\begin{align}\label{eq:def f c}
g(\vic) &:= \prod_{i=1}^n\frac{g_i(\ccc_i)}{\norm\ccc_i}\cdot
\begin{cases}
 1 &\text{ if } \ccc_i+\ccc_j=\OO_K \ \forall i \neq j\\0 & \text{otherwise},
\end{cases} 
,\end{align}
makes the 
ensuing 
estimate 
available,
\begin{equation*}
M_{\vpsi}=\upgamma_0\one_{\upgamma}(\ddd)
\sum_{\substack{\vic \in  \ideals^n\\\norm\ccc_i < \sqrt{Y_i}\\\ccc_i+\ddd\www=\OO_K}}\omega_{\vpsi}(X;\norm\vic)g(\vic) + O_\epsilon(X^{1+\epsilon}).
\end{equation*}

\subsection{Excluding small conjugates}
For $X,Z>0$, 
$w\in\archplaces$
and a separable form $F \in K_w[s,t]$, let
\begin{equation*}
  \mathcal{B}_{F,w}(X;Z) := \big\{(s,t)\in K_w^2 \where \abs{s}_w,\abs{t}_w \leq X^{1/m} \text{ and }
\abs{F(s,t)}_w\leq Z^{1/m}\big\}.  
\end{equation*}

\begin{lemma}\label{lem:simple volume estimate}
  We have 
  \begin{equation*}
    \vol\mathcal{B}_{F,w}(X;Z) \ll_F
    \begin{cases}
      (XZ)^{m_w/m} &\text{ if } 1\leq \deg F < 3,\\
      Z^{2 m_w/(m \deg(F))} &\text{ if } \deg F \geq 3. 
    \end{cases}
  \end{equation*}
\end{lemma}

\begin{proof}
  First, let $\deg F = 1$. The bound claimed in the lemma is obvious if $F$ is proportional to $t$. If $F$ is not proportional to $t$, then the linear transformation $L : K_w^2 \to K_w^2$ given by $L(s,t) = (F(s,t),t)$ is an isomorphism
and thus
  \begin{equation*}
    \vol\mathcal{B}_{F,w}(X;Z) \ll_F \vol\{(s,t)\in K_w^2\where \abs{s}_w \leq Z^{1/m}, \abs{t}_w \ll X^{1/m}\} \ll (XZ)^{m_w/m}.
  \end{equation*}

Next, let us consider the case where $F$ is a quadratic form equivalent to $s^2-t^2$ over $K_w$. Then we can find an invertible linear transformation $L : K_w^2\to K_w^2$ with $F(L(s,t)) = st$, and hence
  \begin{align*}
    \vol\mathcal{B}_{F,w}(X;Z) &\ll_F \vol\{(s,t)\in K_w^2\where \abs{s}_w, \abs{t}_w \ll_F X^{1/m}, \abs{st}_w\leq Z^{1/m}\}\\ &\ll_F X^{m_w/m} + Z^{m_w/m}\log(X) \ll (XZ)^{m_w/m}.
  \end{align*}

 If $F$ is a quadratic form equivalent to $s^2+t^2$ over $K_w=\R$, then we get
 \begin{equation*}
   \vol\mathcal{B}_{F,w}(X;Z) \ll_F \vol\{(s,t)\in\R^2\where s^2+t^2\leq Z^{1/m}\} \ll Z^{1/m} \ll (XZ)^{m_w/m}.
 \end{equation*}
 
 It remains to consider the case where $\deg F \geq 3$. In this case, $F$ is the product of at least three non-proportional linear factors in $\mathbb{C}$
and therefore
\begin{equation*}
V_{w,F}:=\vol\big\{(s,t)\in K_w^2\where \abs{F(s,t)}_w\leq 1\big\} < \infty.
 \end{equation*}
 We procure the validity of 
 \begin{equation*}
   \vol\mathcal{B}_{F,w}(X;Z) \ll \vol(Z^{1/(m\deg(F))}V_{w,F}) \ll_F Z^{2 m_w/(m \deg(F))}. 
 \end{equation*}
\end{proof}
For any non-constant separable form $F\in K_w[s,t]$, let
  \begin{equation*}
    \mathcal{D}_{F,w}^<(X) := \{(s,t)\in X^{1/m}\mathcal{D}\where \abs{F(s_{w},t_{w})}_{w}< 1\}.
  \end{equation*}
Using Lemma \ref{lem:simple volume estimate}
validates the next 
estimate 
  \begin{equation*}
    \vol\mathcal{D}_{F,w}^<(X) \ll_{\mathcal{D}} X^{2-2m_w/m}\cdot 
\vol\mathcal{B}_{F,w}(X,1) \ll_{F} X^{2-2m_w/m}\cdot X^{m_w/m},
\end{equation*}
thus providing the proof of the next lemma.
\begin{lemma} 
For $X\geq 1$ we have
    $\vol\mathcal{D}_{F,w}^<(X) \ll_{\mathcal{D},F} X^{2-m_w/m}$.
\end{lemma} 
For every $w\in\archplaces$
we choose a finite set $\mathcal{H}_w$ of forms in $K_w[s,t]$, whose absolute values we want to prevent from becoming too small. For all $w\in\archplaces$, the set $\mathcal{H}_w$ contains $s$, $t$, and the forms $F_i$ for $1\leq i \leq n$. Additionally, for each form $F_i$ that is of degree $2$ and reducible over $K_w$, we choose a factorization $F_i=G_{i,w}H_{i,w}$ and also include $G_{i,w},H_{i,w}$ in $\mathcal{H}_w$.

Recall the definition of $\mathcal{D}_{\vpsi}(X;\vv)$ in \eqref{eq:def dpsi}. For $\vpsi\in\{0,1\}^n$ and $\vv\in\R^n$, let
\[
\mathcal{D}^*_{\vpsi}(X;\vv):=
\{(s,t)\in \mathcal{D}_{\vpsi}(X;\vv) \where \abs{H_w(s_w,t_w)}_w \geq 1\ \forall w\in\archplaces,\ \forall H_w\in \mathcal{H}_w\}
\]
and
\begin{equation}
\label{eq:ferume}
\omega_\vpsi^*(X;\vv) := \vol\mathcal{D}^*_{\vpsi}(X;\vv).
\end{equation}
We obtain that 
 \begin{equation*}
    |\omega_\vpsi(X;\vv) - \omega_\vpsi^*(X;\vv)|\leq \sum_{w\in\archplaces}\sum_{H_w\in\mathcal{H}_w}\vol\mathcal{D}_{H_w,w}^<(X)
  \end{equation*}
and thus
  \begin{equation*}
    \omega_\vpsi(X;\vv) = \omega_\vpsi^*(X;\vv) + O(X^{2-1/m}).
  \end{equation*}
We can now bring into play the entity
\begin{equation}\label{eq:def cal mpsi}
\mathcal{M}_{\vpsi}:=\sum_{\substack{\vic \in  \ideals^n\\\norm\ccc_i < \sqrt{Y_i}\\\ccc_i+\ddd\www =\OO_K
}} \hspace{-0,4cm} \omega_{\vpsi}^*(X;\norm\vic)g(\vic),
\end{equation}
something which
instantly
permits us to infer the asymptotic relationship
\begin{equation}
\label{han1}
M_{\vpsi}=\upgamma_0\one_{\upgamma}(\ddd)\mathcal{M}_{\vpsi}+O_{\epsilon}(X^{2-1/m+\epsilon}).
\end{equation}

\subsection{Volume computations}
In this section
we provide  
estimates of the correct order of magnitude
regarding 
the volumes $\omega_\vpsi^*(X;\vv)$ appearing in $\mathcal{M}_{\vpsi}$. The assumption 
$c(\mathfrak{F})\leq 3$ will not be used. 
Let us write $d_i:=\deg F_i$ for $1\leq i\leq n$
and
consider, for $q\in\N$ and $T > 0$, the real integral
\begin{equation*}
I_q(T) := \int_{\substack{x_1, \ldots, x_q \geq 1\\x_1\cdots x_q < T}}
\
1
\
\mathrm{d} x_1\cdots \mathrm{d} x_q.
\end{equation*}
One can show that in the range $T\geq 1$
the equality
  \begin{equation*}
    I_{q}(T) = (-1)^q + \sum_{j=1}^q\frac{(-1)^{q-j}}{(j-1)!}T(\log T)^{j-1}
\end{equation*}
holds
via induction coupled with
$I_{q+1}(T) = \int_1^TI_{q}(T/x)\mathrm{d}x$,
thus furnishing the
succeeding result.
\begin{lemma}\label{lem:hyp vol}
There is a polynomial $P_q(T)\in\Q[T]$ of degree $q-1$ and with leading coefficient $1/(q-1)!$
such that for $T\geq 1$
one has 
$I_{q}(T) = TP_{q}(\log T) + (-1)^q$.
\end{lemma}
For $Z\geq 1$ and $1\leq i \leq n$ with $\deg F_i(s,t)\geq 3$
we 
let
  \begin{equation*}
    \mathcal{D}^*_{i}(Z) := \{(s,t)\in K_\infty^2 \where \abs{F_i(s_w,t_w)}_w\geq 1 \text{ for all }w\in\archplaces \text{ an }\norm(F_i(s,t))<Z\}
  \end{equation*}
and 
\begin{equation*}
  \mathcal{D}^*_{s}(Z) := \{s\in K_\infty \where \abs{s_w}_w\geq 1 \text{ for all } w\in\archplaces \text{ and } \norm(s)< Z\}.
\end{equation*}
Letting
$\Omega'\subseteq \archplaces$ be a set of real places,
we
write $\Omega'':=\archplaces\smallsetminus\Omega'$
and 
subsequently define $\mathcal{D}^*_{\Omega',\Omega''}(Z)$ through
\begin{equation*}
   \left\{((s_w,t_w)_{w\in\Omega'},(s_w)_{w\in\Omega''})\in \prod_{w\in\Omega'}K_w^2\times \prod_{w\in\Omega''}K_w \where
  \begin{aligned}
    &\abs{s_w^2+t_w^2}_w\geq 1 \text{ for all }w\in\Omega',\\
    &\abs{s_w}_w\geq 1 \text{ for all }w\in\Omega'',\\
   &\prod_{w\in\Omega'}\abs{s_w^2+t_w^2}_w^{\dg_w}\cdot\prod_{w\in\Omega''}\abs{s_w}^{\dg_w}
    < Z
  \end{aligned}
\right\}.
\end{equation*}

\begin{lemma}\label{lem:technical volumes}
  Let $q:=|\archplaces|$. There are positive constants $c_i, c_s, c_{\Omega',\Omega''}$, such that
  \begin{align*}
\vol\mathcal{D}_s^*(Z) &= c_s I_q(Z),\\
\vol\mathcal{D}_i^*(Z) &= c_i I_q(Z^{2/d_i}),\\
\vol\mathcal{D}^*_{\Omega',\Omega''}(Z) &= c_{\Omega',\Omega''} I_q(Z).
  \end{align*}
\end{lemma}

\begin{proof}
Let
$C = \prod_{v\in\archplaces}(a_v, b_v]\subseteq [0,\infty)^\archplaces$,
$V_{w,i}:=\vol\{(s,t)\in K_w^2\where \abs{F_i(s,t)}_w\leq 1\} 
\hspace{-0,1cm}
< 
\hspace{-0,1cm}
\infty$
and
consider the measurable functions 
  \begin{align*}
    &\Phi_{i} : K_\infty^2 \to [0,\infty)^\archplaces,\quad (s,t) \mapsto (\abs{F_i(s_w,t_w)}_w^{2m_w/{d_i}})_{w\in\archplaces},\\
    &\Phi_{s} : K_\infty \to [0,\infty)^\archplaces,\quad s \mapsto (\abs{s_w}_w^{m_w})_{w\in\archplaces},\\
    &\Phi_{\Omega',\Omega''} : \prod_{w\in\Omega'}K_w^2\cdot\prod_{w\in\Omega''}K_w \to [0,\infty)^\archplaces,
\
((s_w,t_w)_w, (s_w)_w) \mapsto ((\abs{s^2+t^2}_w^{m_w})_{w\in\Omega'},(\abs{s}_w^{m_w})_{w\in\Omega''}).  
  \end{align*}
By homogeneity
we see that 
$  \vol \Phi_i^{-1}(C)$ equals 
\begin{align*}
&\prod_{w\in\archplaces}\vol\{(s_w,t_w)\in K_w^2 \where  a_w < \absv{F_i(s_w,t_w)}^{2m_w/d_i} \leq b_w\}\\
=&\prod_{w\in\archplaces}V_{w,i}(b_w-a_w) = \left(\prod_{w\in\archplaces}V_{w,i}\right)\cdot \vol C.
\end{align*}
In like manner, letting
$V_{w,s}:=\vol\{s\in K_w\where \abs{s}_w\leq 1\} < \infty$
and
   \[V_{w,s^2+t^2}:=\vol\{(s,t)\in K_w^2\where \abs{s^2+t^2}_w\leq 1\},\]
we observe that $V_{w,s^2+t^2}$ is finite if $w$ is a real place
and
\begin{align*}
  \vol \Phi_s^{-1}(C) &= \left(\prod_{w\in\archplaces}V_{w,s}\right)\cdot \vol C,
\\
  \vol \Phi_{\Omega',\Omega''}^{-1}(C) &= \left(\prod_{w\in\Omega'}V_{w,s^2+t^2}\cdot\prod_{w\in\Omega''}V_{w,s} \right)\cdot \vol C.
\end{align*}
This shows that the pushforward measures $\Phi_{i,*}(\vol)$, $\Phi_{s,*}(\vol)$, $\Phi_{\Omega',\Omega'',*}(\vol)$ are constant multiples of the Lebesgue measure on $[0,\infty)^\archplaces$. Let
$\mathcal{H}(T) $ be given by
\begin{equation*}
 \big\{(x_w)_{w\in\archplaces}\where x_w\geq 1 \text{ for all }w \text{ and }\prod_{w\in\archplaces}x_w < T\big\}.
\end{equation*}
Then $\vol\mathcal{H}(T) = I_{q}(T)$,
$\mathcal{D}_i^*(Z)= \Phi_i^{-1}(\mathcal{H}(Z^{2/d_i}))$,
$\mathcal{D}_s^*(Z)= \Phi_s^{-1}(\mathcal{H}(Z))$,
as well as
$
\mathcal{D}_{\Omega',\Omega''}^*(Z) =\Phi_{\Omega',\Omega''}^{-1}(\mathcal{H}(Z))
$,
from which 
the lemma flows immediately.
\end{proof}
For $1\leq i\leq n$,
$1\leq Z_1\leq Z_2$
and
$X\geq 1$
let
\begin{equation*}
  \mathcal{R}_{i}(X;Z_1,Z_2) := \left\{(s,t)\in X^{1/m}\mathscr{D}\where
  \begin{aligned}
    &\abs{H_w(s_w,t_w)}_w\geq 1\ \forall w\in\archplaces\ \forall
    H_w\in\mathcal{H}_w\\ &Z_1 \leq \norm(F_i(s,t))< Z_2
  \end{aligned}
\right\}.
\end{equation*}

\begin{lemma}\label{lem:vol of difference}
Denoting $|\archplaces|$
by $q$ we have  
  \begin{equation*}
    \vol\mathcal{R}_{i}(X;Z_1, Z_2) \ll
    \begin{cases}
      X(I_q(Z_2)-I_q(Z_1)) &\text{ if } d_i = 1\\
      I_q(Z_2^{2/d_i})-I_q(Z_1^{2/d_i}) &\text{ if } d_i \geq 3.
    \end{cases}
  \end{equation*}
  If $d_i = 2$, let $\Omega'$ be the set of  real $w\in\archplaces$ for which $F_i$ is irreducible over $K_w$ and
define
$\Omega'':=\archplaces\smallsetminus\Omega'$. Then
$ \vol\mathcal{R}_{i}(X;Z_1, Z_2)$
is bounded by 
  \begin{equation*}
    \ll\int_{\substack{t_w\in K_w\ \forall w\in\Omega''\\\abs{t_w}_w\geq 1\ \forall w\in \Omega''}}\left(I_{q}(Z_2\prod_{w\in\Omega''}\abs{t_w}_w^{-m_w}) - I_{q}(Z_1\prod_{w\in\Omega''}\abs{t_w}_w^{-m_w})\right) \prod_{w\in\Omega''}\mathrm{d}t_w.
  \end{equation*}
\end{lemma}

\begin{proof}
We deploy Lemma \ref{lem:technical volumes} throughout the proof. Assume first that $d_i\geq 3$. Then
\begin{align*}
  \vol\mathcal{R}_{i}(X;Z_1, Z_2) \ll \vol(\mathcal{D}_i^*(Z_2)\smallsetminus\mathcal{D}_i^*(Z_1)) = c_i(I_q(Z_2^{2/d_i})-I_q(Z_1^{2/d_i})).
\end{align*}
Next, assume that $d_i=1$. Since $F_i$ is not proportional to $t$, the linear transformation $L:K^2\to K^2$ given by $L(s,t)=(F_i(s,t),t)$ is invertible and provides us with the estimate
\begin{align*}
\vol\mathcal{R}_{i}(X;Z_1,Z_2) &\ll \vol\{(s,t)\in K_\infty^2 \where \abs{s_w}_w \geq 1,\ \abs{t_w}_w \ll X^{1/m}\ \forall w\text{ and }Z_1 < \norm(s)\leq Z_2\}\\ &\ll X\vol( \mathcal{D}_s^*(Z_2)\smallsetminus\mathcal{D}_s^*(Z_1))\ll X (I_q(Z_2)-I_q(Z_1)). 
\end{align*}

We are left with the case $d_i=2$. For each $w\in\Omega'$, there is a linear transformation $L_w : K_w^2 \to K_w^2$ such that $F_i(L_w(s,t)) = s^2+t^2$. For $w\in\Omega''$, we have $F_i = G_{i,w}H_{i,w}$ for linear forms $G_{i,w},H_{i,w} \in \mathcal{H}_w$. The linear map $K_w^2 \to K_w^2$, $(s,t) \mapsto (G_{i,w}(s,t),H_{i,w}(s,t))$  has an inverse $L_w$
because $F_i$ is separable. 
We combine all these linear maps to an invertible $\R$-linear map $L = (L_w)_{w\in\archplaces} : K_\infty^2 \to K_\infty^2$, which we apply to obtain
\begin{align*}
  &\vol\mathcal{R}_i(X;Z_1,Z_2) \ll \vol\left\{(s,t)\in K_\infty^2\where 
  \begin{aligned}
    &\abs{s_w^2+t_w^2}_w\geq 1 \text{ for all }w\in\Omega'\\
    &\abs{s_w}_w,\abs{t_w}_w \geq 1 \text{ for all }w\in\Omega''\\
    &Z_1 < \prod_{w\in\Omega'}\abs{s_w^2+t_w^2}_w^{m_w}\prod_{w\in\Omega''}\abs{s_wt_w}_w^{m_w} \leq Z_2
  \end{aligned}
\right\}\\
&=\int_{\substack{t_w\in K_w\ \forall w\in\Omega''\\\abs{t_w}_w\geq 1\ \forall w\in \Omega''}}\vol\left(\mathcal{D}_{\Omega',\Omega''}^*(Z_2\prod_{w\in\Omega''}\abs{t_w}_w^{-m_w}) \smallsetminus \mathcal{D}_{\Omega',\Omega''}^*(Z_1\prod_{w\in\Omega''}\abs{t_w}_w^{-m_w})\right) \prod_{w\in\Omega''}\mathrm{d}t_w.
\end{align*}
\end{proof}

\begin{lemma}\label{lem:volume asympt}
For each
$\vpsi\in\{0,1\}^n$ 
we have 
\begin{equation*}
\omega_{\vpsi}^*(X; (1,\ldots, 1))=X^2\vol(\mathscr{D})+O_\epsilon(X^{2-1/m}+X^{3/2+\epsilon}).
\end{equation*}
\end{lemma}

\begin{proof}
Let us begin by observing that 
\begin{align*}
  |X^2\vol(\mathscr{D})-\omega_{\vpsi}^*(X;(1,\ldots, 1)
)| &\ll \sum_{w\in\archplaces}\sum_{H_w\in\mathcal{H}_w}\vol\mathcal{D}_{H_w,w}^<(X) +\sum_{i=1}^n\vol\mathcal{R}_{i}(X;1,\sqrt{Y_i}\norm\www_i)\\
&\ll X^{2-1/m} + \sum_{i=1}^n\vol\mathcal{R}_i(X;1,\sqrt{Y_i}\norm\www_i).
\end{align*}
We now use Lemma \ref{lem:vol of difference} and Lemma \ref{lem:hyp vol} to estimate the $\vol\mathcal{R}_i(X;1,\sqrt{Y_i}\norm\www_i)$. If $d_i = 1$, then
$\vol\mathcal{R}_i(X;1,\sqrt{Y_i}\norm\www_i) \ll X \sqrt{Y_i}^{1+\epsilon}\ll X^{3/2+\epsilon}$,
while, if
$d_i \geq 3$, we acquire 
\begin{equation*}
  \vol\mathcal{R}_i(X;1,\sqrt{Y_i}\norm\www_i) \ll Y_i^{1/d_i+\epsilon}\ll X^{1+\epsilon}.
\end{equation*}
In the remaining case,
$d_i = 2$, we get
\begin{align*}
  \vol\mathcal{R}_i(X;1,\sqrt{Y_i}\norm\www_i) \ll \int_{\substack{t_w\in K_w\ \forall w\in\Omega''\\1\leq \abs{t_w}_w\ll \sqrt{Y_i}\ \forall w\in \Omega''}} \frac{\sqrt{Y_i}^{1+\epsilon}}{\prod_{w\in\Omega''}\abs{t_w}_w^{m_w}}\prod_{w\in\Omega''}\mathrm{d}t_w \ll \sqrt{Y_i}^{1+\epsilon} \ll X^{1+\epsilon}.
\end{align*}
\end{proof}
For a function $\omega : \R^n \to \R$ and $1\leq i\leq n$, we write $\Delta_{i}\omega(\vv) := \omega(\vv+\mathbf{e}_i) - \omega(\vv)$, where $\mathbf{e}_i$ is the $i$-th vector in the standard basis of $\R^n$. 
\begin{lemma}\label{lem:volume partial}
Let $\vpsi\in\{0,1\}^n$,  $1\leq i\leq n$
and $\b{v} \in \R^n$
be given such that
$v_j\in [0,\infty)$ for all $j\neq i$. Then $\omega_{\vpsi}^*(X;\vv)$, considered as a function of $v_i$, is non-increasing 
and
satisfies
\begin{equation}\label{eq:omega differences}
\Delta_{i}\omega_{\vpsi}^*(X;\vv) \ll X^{1+\epsilon}
  \begin{cases}
    \ X^{\frac{1}{2}} &\text{if } d_i=1, \\
    v_i^{\frac{2}{d_i}-1} & \text{otherwise}, \end{cases}
\end{equation}
in the interval $1\leq v_i\leq \sqrt{Y_i}$, with the implied constant independent of $\vv$ and $X$.
\end{lemma}

\begin{proof}
Monotonicity is obvious. Let us prove the estimate \eqref{eq:omega differences}. If $\psi_i=0$, then $\omega_{\vpsi}^*(X; \vv)$ is constant in $v_i$. Let $\psi_i=1$, then 
\begin{align*}
  |\omega_{\vpsi}^*(X; \vv + \ve_i) - \omega_{\vpsi}^*(X; \vv)| &\ll \mathcal{R}_i(X;\sqrt{Y_i}\norm\www_iv_i,\sqrt{Y_i}\norm\www_i(v_i+1)).
\end{align*}
Using Lemma \ref{lem:vol of difference} and the mean value theorem to bound the latter quantity, we obtain in
the
case $d_i=1$ that, for some $\tilde{v}_i \in [v_i,v_i+1]$,
\begin{align*}
  \Delta_i\omega_{\vpsi}^*(X;\vv) &\ll \frac{\partial}{\partial V} (X I_q(\sqrt{Y_i}\norm\www_iV))|_{V=\tilde{v}_i} \ll X\sqrt{Y_i} \frac{\partial}{\partial V}\left(VP_q(\log(\sqrt{Y_i}\norm\www_iV))\right)|_{V=\tilde{v}_i}\\ &\ll X^{3/2+\epsilon}.
\end{align*}
When $d_i\geq 3$, we get
\begin{align*}
  \Delta_i\omega_{\vpsi}^*(X;\vv) &\ll \frac{\partial}{\partial V} I_q((\sqrt{Y_i}\norm\www_iV)^{2/d_i})|_{V=\tilde{v}_i} \ll Y_i^{1/d_i}\frac{\partial}{\partial V}V^{2/d_i}P_q(2/d_i\log(\sqrt{Y_i}\norm\www_iV))|_{V=\tilde{v}_i}\\ &\ll X_i^{1+\epsilon} \tilde{v}_i^{2/d_i-1}.
\end{align*}
When 
$d_i=2$, the quantity
$\Delta_i\omega_{\vpsi}^*(X;\vv)$ 
is
\begin{align*}
 \ll \int_{\substack{t_w\in K_w\ \forall w\in\Omega''\\\abs{t_w}_w\geq 1\ \forall w\in \Omega''}}I_{q}(\sqrt{Y_i}\norm\www_i(v_i+1)\prod_{w\in\Omega''}\abs{t_w}_w^{-m_w}) - I_{q}(\sqrt{Y_i}\norm\www_iv_i\prod_{w\in\Omega''}\abs{t_w}_w^{-m_w})\prod_{w\in\Omega''}\mathrm{d}t_w.
\end{align*}
The integrand is zero, unless $\prod_{w\in\Omega''}\abs{t_w}_w^{\dg_w}\leq \sqrt{Y_i}\norm\www_i(v_i+1)$. In that case, the mean value theorem allows us to find for any $(t_w)_w$ a number $\tilde{v}_i \in (v_i,v_i+1)$, such that the integrand is
\begin{align*}
  & \frac{\partial}{\partial V}\left( I_q(\sqrt{Y_i}\norm\www_iV\prod_{w\in\Omega''}\abs{t_w}_w^{-m_w})\right)\big|_{V=\tilde{v}_i} \\
=\ &\frac{\partial}{\partial V}\left(\sqrt{Y_i}\norm\www_iV\prod_{w\in\Omega''}\abs{t_w}_w^{-m_w}P_q(\log(\sqrt{Y_i}\norm\www_iV\prod_{w\in\Omega''}\abs{t_w}_w^{-m_w}))\right)\big|_{V=\tilde{v}_i}\\
\ll &\sqrt{Y_i}X^\epsilon\prod_{w\in\Omega''}\abs{t_w}_w^{-m_w}\ll X^{1+\epsilon}\prod_{w\in\Omega''}\abs{t_w}_w^{-m_w}.
\end{align*}
This shows that
$\Delta_i\omega_{\vpsi}^*(X;\vv) \ll X^{1+\epsilon}$,
which concludes our proof.
\end{proof}

\subsection{The ending moves towards
Theorem~\ref{thm:main 2}}
\label{s:ending}
We are now ready to estimate the sum 
$\mathcal{M}_{\vpsi}$
that was 
introduced in~\eqref{eq:def cal mpsi}.
\begin{lemma}
  Let $\delta := \max_{1\leq i\leq n}\{4+8m\deg F_i\}$. For any $0\leq i \leq n$, there are functions 
${\upgamma^{(i)}}, {\updelta_1^{(i)}}, \ldots, {\updelta_i^{(i)}} \in \mathcal{Z}_K$, and a positive constant $\mu^{(i)}$, such that
  \begin{align}
\nonumber    \mathcal{M}_{\vpsi}&=\mu^{(i)}\one_{\upgamma^{(i)}}(\ddd)\hspace{-0.3cm}\sum_{\substack{\norm\ccc_1\leq \sqrt{Y_1}\\\ccc_1+\ddd\www=\OO_K}}\hspace{-0.2cm}\frac{\rho_1(\ccc_1)\one_{\updelta_1^{(i)}}(\ccc_1)}{\norm\ccc_1}
\cdots\hspace{-0.7cm} \sum_{\substack{\norm\ccc_i\leq \sqrt{Y_i}\\\ccc_i+\ccc_1\cdots \ccc_{i-1}\ddd\www=\OO_K}}\hspace{-0.7cm}\frac{\rho_i(\ccc_i)\one_{\updelta_i^{(i)}}(\ccc_i)}{\norm\ccc_i}\omega_{\vpsi}^*(X;(\norm\ccc_1,\ldots,\norm\ccc_i,1,\ldots,1))\\ &+ O_\epsilon(\norm\ddd^\epsilon X^{2-1/\delta+\epsilon}).\label{eq:mpsi expanded}    
  \end{align}
\end{lemma}

\begin{proof}
For $i=n$
our lemma holds
with vanishing error term by the definition
of
$g$ in \eqref{eq:def f c}. 
We proceed by
backward induction
from $i$ to $i-1$.
Lemma \ref{lem:twisted sum} provides the existence of $\beta^{(i)}>0$ and 
${\upgamma'^{(i)}}\in\mathcal{Z}_K$ such that, for all $U\geq 1$,
\begin{equation}\label{eq:twisted sum application}
  \sum_{\substack{\norm\ccc_i\leq U\\\ccc_i+\ccc_1\cdots \ccc_{i-1}\ddd\www=\OO_K}}\hspace{-0.7cm}\frac{\rho_i(\ccc_i)\one_{\updelta_i^{(i)}}(\ccc_i)}{\norm\ccc_i} = \beta^{(i)}\one_{\upgamma'^{(i)}}(\ccc_1\cdots\ccc_{i-1}\ddd) + O_\epsilon(\norm(\ccc_1\cdots\ccc_{i-1}\ddd)^\epsilon U^{-1/(2\lambda)+\epsilon}),
\end{equation}
where $\lambda = 1+2m\deg F_i$. Indeed, the hypotheses of Lemma \ref{lem:twisted sum} are satisfied by Lemma \ref{lem:artinakos} and Hensel's lemma, once we ensure that $\www_0$, and hence $\www$, is divisible by enough small prime ideals.

We write $\omega(\theta):=\omega_\vpsi^*(X; (\norm\ccc_1,\ldots,\norm\ccc_{i-1},\theta,1,\ldots,1))$. Assume first that $\deg F_i=1$. In this case, the bounds~\eqref{eq:omega differences}
and~\eqref{eq:twisted sum application} allow us to apply Lemma \ref{lem:abel discrete} with
$A= 1/(2\lambda)$, $B= 0$,
\[
  M \ll_\epsilon \norm(\ccc_1\cdots\ccc_{i-1}\ddd)^\epsilon X^\epsilon
 \
\
\text{and}
\
\
 Q \ll_\epsilon  X^{3/2+\epsilon}  
,\]
thus leading to
\begin{align}\label{eq:last sum estimate}
\sum_{\substack{\norm\ccc_i\leq \sqrt{Y_i}\\\ccc_i+\ccc_1\cdots \ccc_{i-1}\ddd\www=\OO_K}}\hspace{-0.7cm}\frac{\rho_i(\ccc_i)\one_{\updelta_i^{(i)}}(\ccc_i)}{\norm\ccc_i}\omega(\norm\ccc_i) &= \beta^{(i)}\one_{\upgamma^{(i)'}}(\ccc_1\cdots\ccc_{i-1}\ddd)\omega(1)\nonumber\\ &+ O_\epsilon\left(\norm(\ccc_1\cdots\ccc_{i-1}\ddd)^\epsilon X^{2-1/(4\lambda)+\epsilon}\right). 
\end{align}
If $\deg F_i\geq 2$, we use Lemma \ref{lem:abel discrete} with the same bounds for $M,A$ and
\[
  Q\ll_\epsilon X^{1+\epsilon},
\quad B = 1-2/(\deg F_i)
\]
to obtain an estimate identical to~\eqref{eq:last sum estimate}. 
Injecting this in \eqref{eq:mpsi expanded} proves our claim for $i-1$.
\end{proof}
The case $i=0$ of the last lemma shows that 
$\mathcal{M}_\psi = 
\mu^{(0)} \one_{\gamma^{(0)}}(\ddd)\vol\mathcal{D}X^2 + O(\norm\ddd^\epsilon X^{2-1/\delta+\epsilon})
$.
Conjuring up~\eqref{han1} and Lemma \ref{cramped} completes the undertaking of validating Theorem \ref{thm:main 2}.
\bibliographystyle{amsalpha}
\bibliography{fibration}
\end{document}